\documentclass[10pt]{amsart}
\usepackage{amsmath,amsthm,amsfonts,amscd,flafter,epsf,amssymb,manfnt,epsdice,wasysym}
\usepackage{epsfig}
\usepackage{amsfonts}
\usepackage{amssymb}
\usepackage{amsmath}
\usepackage{amsthm}
\usepackage{latexsym}
\usepackage{amscd}
\usepackage{flafter}
\usepackage{epsf,color}
\usepackage{tikz,tikz-cd}
\input{diagram}

\newcommand{\colvec}[2][.8]{%
  \scalebox{#1}{%
    \renewcommand{\arraystretch}{.8}%
    $\begin{pmatrix}#2\end{pmatrix}$%
  }
}
\newcommand{\lift}[2]{%
\setlength{\unitlength}{1pt}
\begin{picture}(0,0)(0,0)
\put(0,{#1}){\makebox(0,0)[b]{${#2}$}}
\end{picture}
}
\newcommand{\lowerarrow}[1]{%
\setlength{\unitlength}{0.03\DiagramCellWidth}
\begin{picture}(0,0)(0,0)
\qbezier(-28,-4)(0,-18)(28,-4)
\put(0,-14){\makebox(0,0)[t]{$\scriptstyle {#1}$}}
\put(28.6,-3.7){\vector(2,1){0}}
\end{picture}
}
\newcommand{\upperarrow}[1]{%
\setlength{\unitlength}{0.03\DiagramCellWidth}
\begin{picture}(0,0)(0,0)
\qbezier(-28,12)(0,25)(28,12)
\put(0,20){\makebox(0,0)[b]{$\scriptstyle {#1}$}}
\put(28.6,11.7){\vector(2,-1){0}}
\end{picture}
}
\newcommand{\strarrow}[1]{%
\setlength{\unitlength}{0.03\DiagramCellWidth}
\begin{picture}(0,0)(0,0)
\qbezier(-22,5)(0,5)(22,5)
\put(0,6){\makebox(0,0)[b]{$\scriptstyle {#1}$}}
\put(-22,4.8){\vector(-2,0){0}}
\end{picture}
}

\newcommand{\Hcal}{\mathcal{H}}
\newcommand{\ok}{\ov{k}}
\newcommand{\oc}{\ov{c}}
\newcommand{\co}{\colon\thinspace}

\newcommand{\CFDT}{\mathrm{CFD}}
\newcommand{\CT}{\mathrm{C}}
\newcommand{\KT}{\mathrm{K}}
\newcommand{\pphi}{\mathfrak{f}}
\newcommand{\Ff}{\mathfrak{f}}
\newcommand{\pphibar}{\overline{\pphi}}

\newcommand{\Pbb}{\mathbb{P}}
\newcommand{\ovl}{\overline}

\newcommand{\Dd}{\mathfrak{D}}
\newcommand{\Ht}{\mathrm{H}}
\newcommand{\HFT}{\mathrm{HF}}
\newcommand{\CFT}{\mathrm{CF}}
\newcommand{\HFKT}{\mathrm{HFK}}

\newcommand{\rank}{\mathrm{rnk}}

\newcommand{\taub}{\Xi}

\newcommand{\CFKT}{\mathrm{CFK}}
\newcommand{\ov}{\widehat}

\newcommand{\Sig}{\Sigma}
\newcommand{\ra}{\rightarrow}

\newcommand\Dual{\mathcal D}
\newcommand\Duality\Dual

\newcommand\RelSpinC{\underline{\SpinC}}
\newcommand\relspinc{s}

\newcommand\x{\mathbf x}

\newcommand\z{\mathbf z}

\newcommand\y{\mathbf y}

\newcommand\ModSphere{\ModFlow\left({\mathbb S}\longrightarrow
\Sym^{g-1}(\Sigma_{1})\times \Sym^2(\Sigma_{2})\right)}
\newcommand\ModSpheres\ModSphere

\newcommand\UnparModSp{\widehat \ModSp}
\newcommand\UnparModFlow\UnparModSp

\newcommand\spin{\mathfrak s}

\newcommand{\spinc}{\mathfrak s}

\newcommand{\spinct}{\mathfrak t}

\newcommand\ModMaps{\mathcal M}
\newcommand\ModSp\ModMaps

\newcommand\Ta{{\mathbb T}_{\alpha}}
\newcommand\Tb{{\mathbb T}_{\beta}}

\newcommand\alphas{\mbox{\boldmath$\alpha$}}

\newcommand\betas{\mbox{\boldmath$\beta$}}

\newcommand\spincrel\relspinc

\newtheorem{thm}{Theorem}[section]
\newtheorem{prop}[thm]{Proposition}
\newtheorem{cor}[thm]{Corollary}
\newtheorem{conj}[thm]{Conjecture}
\newtheorem{lem}[thm]{Lemma}

\newtheorem{defn}[thm]{Definition}

\def\endproof{\relax\ifmmode\expandafter\endproofmath\else
  \unskip\nobreak\hfil\penalty50\hskip.75em\hbox{}\nobreak\hfil\bull
  {\parfillskip=0pt \finalhyphendemerits=0 \bigbreak}\fi}
\def\endproofmath$${\eqno\bull$$\bigbreak}
\def\bull{\vbox{\hrule\hbox{\vrule\kern3pt\vbox{\kern6pt}\kern3pt\vrule}\hrule}}

\newcommand{\Z}{\mathbb{Z}}

\newcommand{\Ker}{\mathrm{Ker}}

\newcommand{\Coker}{\mathrm{Coker}}

\newcommand{\Image}{\mathrm{Im}}

\newcommand{\ModSWfour}{\mathcal{M}}
\newcommand{\ModFlow}{\ModSWfour}

\newcommand{\SpinC}{{\mathrm{Spin}}^c}

\newcommand\abuts\Rightarrow
\newcommand\Sym{\mathrm{Sym}}

\newcommand{\Hbb}{{\mathbb{H}}}

\newcommand{\Fbb}{\mathbb{F}}

\newcommand{\lra}{\longrightarrow}

\newcommand{\Mod}{\mathcal{M}}

\begin{document}

\title[Bordered Floer homology and incompressible tori]
{Bordered Floer homology and existence of incompressible tori in homology spheres}%
\author{Eaman Eftekhary}%
\address{School of Mathematics, Institute for Research in Fundamental Sciences (IPM),
P. O. Box 19395-5746, Tehran, Iran}%
\email{eaman@ipm.ir}%

\thanks{}%
%\subjclass{}%
\keywords{Floer homology, splicing, incompressible torus}%

%\date{December 2005}%
%\dedicatory{}%
%\commby{}%
\maketitle
% ----------------------------------------------------------------
\begin{abstract}
Let $K$ denote a knot inside the homology sphere $Y$. 
The zero-framed longitude of $K$ gives the complement of $K$ 
in $Y$ the structure of a bordered three-manifold, which may be denoted by 
$Y(K)$. We compute the bordered Floer 
complex $\ov{\mathrm{CFD}}(Y(K))$ in terms of the knot Floer complex 
associated with $K$. As a corollary,
we show that if a homology sphere has the same Heegaard Floer homology
as $S^3$  it does not contain any incompressible tori. 
Consequently, if $Y$ is an irreducible  homology sphere $L$-space then $Y$ is either 
$S^3$, or  the Poicar\'e sphere $\Sig(2,3,5)$, or it is hyperbolic.
\end{abstract}
%\tableofcontents
\section{Introduction}\label{sec:introduction}
\subsection{Background and the main results }
Heegaard Floer homology, defined by  Ozsv\'ath and Szab\'o \cite{OS-3mfld},
has been  powerful in extracting topological properties of three-manifolds.
Although there is a variety of $L$-spaces- the three-manifolds 
with most simple Heegaard Floer homology-
in rare cases homology spheres have the Heegaard Floer homology of $S^3$.
 The Poincar\'e sphere $\Sig(2,3,5)$ is  an example 
 of an irreducible homology sphere with
$\widehat{\mathrm{HF}}(\Sig(2,3,5))=\widehat{\mathrm{HF}}(S^3)=\Z$.
It is thus not true in general, that Heegaard Floer homology
is capable to distinguish $S^3$ from other homology spheres.
 However, a conjecture of Ozsv\'ath and Szab\'o
predicts that the $\Sig(2,3,5)$ is the only non-trivial example of an irreducible
homology sphere with trivial Heegaard Floer homology. 
In this paper, we will address the case of a $3$-manifold which contains  
an incompressible torus. Let $\Fbb$ denote the field $\Z/2\Z$ %with two elements
throughout this paper.

\begin{thm}\label{thm:main-intro}
If a homology  sphere $Y$
 contains an incompressible torus  $$\widehat{\mathrm{HF}}(Y;\Fbb)\neq \Fbb=\widehat{\mathrm{HF}}(S^3;\Fbb).$$
\end{thm}

Together with Thurston's geometrization conjecture, now a theorem of Perelman  
(see \cite{Thurston, Per}, also \cite{ Morgan-Tian1, Morgan-Tian2}), 
this result reduces the study of Ozsv\'ath-Szab\'o conjecture to the homology spheres
which are either Seifert fibered or hyperbolic.
It may be shown (see \cite{Raif}, also \cite{Ef-Seifert})
that Poincar\'e sphere and the standard sphere are the only Seifert fibered homology spheres
with trivial Heegaard Floer homology. Thus, Ozsv\'ath-Szab\'o conjecture is reduced to the
following.
\begin{conj}\label{conj} 
 If the homology sphere $Y$ is hyperbolic 
 $\widehat{\mathrm{HF}}(Y;\Fbb)\neq \Fbb$.
\end{conj}

The reduced Khovanov 
homology of a knot $K$ inside the standard three-sphere
is related to the Heegaard Floer homology of the double cover 
of the standard sphere branched over $K$ \cite{OS-branched}.
Ozsv\'ath-Szab\'o conjecture (or Conjecture~\ref{conj}) thus implies that 
the reduced Khovanov homology (and thus Khovanov homology) detects the unknot;
a theorem of Kronheimer and Mrowka \cite{KM}.
The result of this paper reproves a few special cases of the aforementioned theorem.
A knot $K\subset S^3$ is 
$\pi$-\emph{hyperbolic} if $S^3-K$ admits a Riemannian
metric with constant negative curvature which 
becomes singular folding with an angle $\pi$ around $K$.
%In particular, $\pi$-hyperbolic knots and hyperbolic.
\begin{cor}
Suppose that for a knot $K\subset S^3$  one of the following is true:
\begin{itemize}
\item $K$ is not $\pi$-hyperbolic.
\item $K$ is a prime satellite knot.
\end{itemize}
Then  the rank of the reduced Khovanov homology 
$\widetilde{\mathrm{Kh}}(K)$ is greater than $1$.
\end{cor}

\subsection{Bordered Floer homology for a knot complement}\label{subsec:re-statement}
The proof of Theorem~\ref{thm:main-intro} rests heavily on a 
construction  of the bordered Floer module $\ov\CFDT(Y(K))$ associated with the
 complement of a knot $K\subset Y$ using the knot Floer complex $\CFKT(Y,K)$. 
Consider a doubly pointed 
Heegaard diagram $(\Sig,\alpha,\beta;u,v)$ for 
$K$. The markings $u$ and $v$ give the map
$$\spinc=\spinc_{u,v}:\mathbb{T}_{\alpha}\cap\mathbb{T}_{\beta}
\lra \RelSpinC(Y,K)$$
where $\spinc(\x)$ denotes the relative $\SpinC$ class assigned to $\x$
in the sense of \cite{Ni},  which is defined by assigning a nowhere vanishing vector field 
on $Y-\mathrm{nd}(K)$ to $\x$ which is tangent to the boundary. 
Multiplying the vector fields by $-1$ gives a map 
$$J:\RelSpinC(Y,K)\lra \RelSpinC(Y,K)$$
and the map $\spinc\mapsto c_1(\spinc)=\spinc-J(\spinc)\in\Ht^2(Y,K;\Z)$
gives an identification of $\RelSpinC(Y,K)$ with $\Z$, which will be 
implicit throughout this paper. Let 
$$C=\left\langle[\x,i,j]\ \big|\ 
\x\in\mathbb{T}_{\alpha}\cap\mathbb{T}_{\beta},\ 
\spinc(\x)-i+j=0\right\rangle_\Z$$
 denote the $\Z\oplus \Z$ filtered chain complex associated 
with $K$.
Following \cite{OS-surgery} we may consider the sub-modules
$$ C\{i=a,j=b\}, C\{i=a,j\leq b\}\ \ \text{and}\ \ 
 C\{i\leq a,j=b\}\  %\text{etc.},
\ \ \ a,b\in\Z\cup\{\infty\}$$
 with the induced structure as a chain complex. Set 
 $C\{i=a\}$ to be the chain complex $C\{i=a,j\leq \infty\}$ and 
 $C\{j=b\}$ to be $C\{i\leq \infty,j=b\}$.
 For any relative $\SpinC$ class $\spinc\in\Z=\RelSpinC(Y,K)$ let
\begin{displaymath}
\begin{split}
&i_n^\spinc=i_n^\spinc(K):C\{i\leq \spinc,j=0\}\oplus C\{i=0,j\leq n-\spinc-1\}\lra 
C\{j=0\}\\
&i_n^\spinc\left([\x,i,0],[\y,0,j]\right):=[\x,i,0]+\taub[\y,0,j],
\end{split}
\end{displaymath}
where $\taub:C\{i=0\}\ra C\{j=0\}$ is the chain homotopy equivalence
corresponding to the Heegaard moves which change the diagram
$(\Sig,\alphas,\betas;u)$ to $(\Sig,\alphas,\betas;v)$.
Let $Y_n(K)$ denote the three-manifold obtained from $Y$ 
by $n$-surgery on $K$ and let 
$K_n$ denote the corresponding knot 
inside   $Y_n(K)$ which is determined by the aforementioned surgery.
\begin{prop}\label{prop:surgery-formula-intro}
The homology of the mapping cone $M(i_n^\spinc)$ gives 
$$\Hbb_n(K,\spinc)=\ov\HFKT(Y_n(K), K_n,\spinc).$$ 
\end{prop}
Note that $M(i_0^\spinc)$ is a sub-complex of both $M(i_1^\spinc)$
and $M(i_1^{\spinc+1})$.
We denote the embedding maps by  $F_\infty^\spinc=F_\infty^\spinc(K)$ and 
$\ovl{F}_\infty^{\spinc+1}=\ovl{F}_\infty^{\spinc+1}(K)$, respectively. The
quotient of $M(i_1^\spinc)$ by $F_\infty^\spinc\left(M(i_0^\spinc)\right)$
 is isomorphic to 
$$\ov\CFKT(K,\spinc)\simeq C\{i=0,j=-\spinc\}.$$ 
Denote the quotient map by $F_0^\spinc=F_0^\spinc(K)$. 
Similarly, define the quotient map $\ovl{F}_0^\spinc=\ovl{F}_0^\spinc(K)$
from $M(i_1^\spinc)$ to 
$M(i_1^\spinc)/\Image(\ovl{F}_\infty^\spinc)$.
The short exact sequences 
%\begin{equation}\label{eq:SES-2knots}
\begin{displaymath}
\begin{split}
&\begin{diagram}
0&\rTo & M\left(i_0^\spinc\right)&\rTo{F_\infty^\spinc}& 
M\left(i_1^\spinc\right)&\rTo{F_0^\spinc\ \ }&\ov\CFKT(K,\spinc)
&\rTo &0
\end{diagram}\ \ \ \ \ \text{and}\\
&\begin{diagram}
0&\rTo & M\left(i_0^{\spinc-1}\right)&\rTo{\ovl{F}_\infty^\spinc}& 
M\left(i_1^\spinc\right)&\rTo{\ovl{F}_0^\spinc\ \ }&\ov\CFKT(K,\spinc)
&\rTo &0
\end{diagram}
\end{split}
\end{displaymath}
%\end{equation}
give the following two homology exact triangles
\begin{equation}\label{eq:LES-2knots}
\begin{diagram}
\Hbb_0(\spinc)&&\rTo{\pphi_\infty^\spinc}&& \Hbb_1(\spinc)&&
\Hbb_0(\spinc-1)&&\rTo{\pphibar_\infty^\spinc}&& \Hbb_1(\spinc)\\
&\luTo{\pphi_1^\spinc}&&\ldTo{\pphi_0^\spinc}
&&&&\luTo{\pphibar_1^\spinc}&&\ldTo{\pphibar_0^\spinc}&\\
&&\Hbb_\infty(\spinc)&&&\text{and}&&&\Hbb_\infty(\spinc)&&
\end{diagram},
\end{equation}
where $\Hbb_\bullet(\spinc)=\Hbb_\bullet(K,\spinc)$. We let 
\begin{displaymath}
C_\bullet(K)=\bigoplus_{\spinc\in\Z}C_\bullet(K,\spinc)\ \ \text{and}\ \ 
\Hbb_\bullet(K)=\bigoplus_{\spinc\in\Z}\Hbb_\bullet(K,\spinc),\ \ \ 
\bullet\in\{0,1,\infty\}
\end{displaymath}
where $C_\bullet(K,\spinc)=M(i_\bullet^\spinc)$ for $\bullet=0,1$ and 
$C_\infty(K,\spinc)=C\{i=\spinc,j=0\}$.
Denote the differential of 
$C_\bullet(K)$ by $d_\bullet$ for $\bullet\in\{0,1,\infty\}$. Set
$M(K)=C_0(K)\oplus C_1(K)$ and $L(K)=C_1(K)\oplus C_\infty(K)$.
 Let $F_\bullet=F_\bullet(K)$ denote the map obtained by 
putting all $F_\bullet^\spinc$ together. These maps will be called the {\emph{bypass
homomorphisms}}.
\\

The zero framed longitude of $K$ and its meridian give a parametrization of the 
the torus boundary $-T^2$ of $Y\setminus \mathrm{nd}(K)$. The corresponding bordered three-manifold is denoted by $Y(K)$. 
 A differential graded algebra $\mathcal{A}(T^2,0)$ is associated with $T^2$.
 The bordered Floer module $\ov\CFDT(Y(K))$  
is then a module over the differential graded algebra 
$\mathcal{A}(T^2,0)$. 
Following the notation of Subsection~4.2 from 
\cite{LOT-E},  $\mathcal{A}(T^2,0)$ is generated, as a module over $\Fbb$, 
by the idempotents $\imath_0$ and $\imath_1$, and the chords 
$\rho_1,\rho_2,\rho_3,\rho_{12}=\rho_{1}\rho_2,\rho_{23}=\rho_2\rho_3$ and 
$\rho_{123}=\rho_1\rho_2\rho_3$;
\begin{displaymath}
\begin{split}&\\
&\mathcal{A}(T^2,0)=
\Big\langle\begin{diagram}[w=3em]
\imath_0 \ \bullet\ & \upperarrow{\rho_1}
\lift{-2}
{\strarrow{\rho_2}}
\lowerarrow{\rho_3} &\  \bullet\ \imath_1
\end{diagram}\Big\rangle/\left(\rho_2\rho_1=\rho_3\rho_2=0\right).\\
&
\end{split}
\end{displaymath}  

\begin{thm}\label{thm:BFH}
With the above notation,
 the bordered Floer complex $\ov\CFDT(Y(K))$ is  quasi-isomorphic to the 
 left module over the differential graded algebra $\mathcal{A}(T^2,0)$, which is generated 
 by $\imath_0.L(K)$ and $\imath_1. M(K)$ and is equipped with the differential 
 $\partial\co \ov\CFDT(Y(K))\ra \ov\CFDT(Y(K))$ defined by 
 \begin{equation}\label{eq:differential}
\partial \colvec{
\x\\ \y }=
\begin{cases}
\colvec{
d_0(\x)\\ \ovl{F}_\infty(\x)+d_1(\y)}
+\rho_2.\colvec{
0\\ \x}&\text{if }\colvec{\x\\ \y}\in M(K)\\
\\
\colvec{
d_1(\x)\\ F_0(\x)+d_\infty(\y)}
+\colvec{
\rho_1 F_\infty(\x)\\ \rho_3\ovl{F}_0(\y)+\rho_{123}\ovl{F}_0(F_\infty(\x))}
&\text{if }\colvec{\x\\ \y }\in L(K)\\
\end{cases} 
 \end{equation}
 \end{thm}

%\newpage
\section{Surgery on null-homologous knots}\label{sec:surgery}
\subsection{A triangle of chain maps}
By a Heegaard $n$-tuple we mean the data 
$$(\Sig,\alphas_1,...,\alphas_n;u_1,...,u_r)$$
where $\Sig$ is a Riemann surface of genus $g$, each $\alphas_i$ is $g$-tuples of disjoint
simple closed curves for $i=1,...,n$ and $u_j$ are markings in 
$\Sig-\sqcup_{i=1}^n \alphas_i$. Let $\mathbb{T}_{\alpha_i}\subset 
\Sym^g(\Sig)$ denote the  torus associated with $\alphas_i$ and let
$\x_i\in\mathbb{T}_{\alpha_i}\cap\mathbb{T}_{\alpha_{i+1}}$ for $i=1,...,n-1$ and 
$\x_n \in\mathbb{T}_{\alpha_1}\cap\mathbb{T}_{\alpha_{n}}$ be $n$ intersection 
points. Let $\pi_2(\x_1,...,\x_n)$ denote the set of homotopy classes of 
$n$-gons connecting $\x_1,...,\x_n$ and define $\pi_2^j(\x_1,...,\x_n;u_1,...,u_r)$
to be the subset of $\pi_2(\x_1,...,\x_n)$ which consists of the classes with Maslov 
index $j$ which have zero intersection number with the 
codimension one sub-varieties $L_{u_1},...,L_{u_r}$ of $\Sym^g(\Sig)$
which correspond to the markings $u_1,...,u_r$.\\

Let $K\subset Y$ be a knot inside a homology sphere $Y$.
Consider a Heegaard diagram  
$$H=(\Sig,\alphas=\{\alpha_1,...,\alpha_g\},\betas=\{\beta_1,...,
\beta_g\},p_\infty)$$
for the pair $(Y,K)$, where $\beta_g$ 
corresponds to the meridian of $K$ and  the marking $p_\infty$ is
placed on $\beta_g$, so that  putting a pair of marked points near $p_\infty$ and on the
two sides of $\beta_g$ we obtain a doubly pointed Heegaard diagram
for  $K$. Suppose that $\lambda$
represents a zero-framed longitude for the knot $K$.
Let $\lambda_n$ be a small perturbation of the juxtaposition $\lambda+n\beta_g$
and $\beta_i^n$ denote a small Hamiltonian isotope of $\beta_i$ for $i=1,...,g-1$.
The Heegaard diagram
$$H_n=(\Sig,\alphas,\betas_n=\{\beta_1^n,...,\beta_{g-1}^n,\lambda_n\},p_n)$$
 gives a diagram for $(Y_n(K),K_n)$, where
$p_n$ is a marked point at the intersection of $\lambda_n$ and $\beta_g$. 
With the integers $m>n\geq 0$ fixed, we assume that 
$\lambda_n$ and $\lambda_{n+m}$ intersect each other in $m$ transverse points, 
and that for an intersection point $q$ of these latter curves the points 
$q,p_n,p_{m+n}$ are the vertices of a triangle $\Delta$, which is one of the 
connected components in 
$\Sig- \left(\alphas\cup \betas\cup\{\lambda_n,\lambda_{n+m}\}\right)$.
From the $4$ quadrants which have $q$ as a corner two of them  belong to the neighbors
of $\Delta$. Place a pair of markings $u$ and $v$ in these two quadrants, and use them 
as the punctures in the following discussion.\\

 Fix a relative $\SpinC$ class 
 $\spinc\in\RelSpinC(Y,K)%=\RelSpinC(Y_{n}(K),K_{n})=\RelSpinC(Y_{n+m}(K),K_{n+m})
 =\Z$.
The  complex  associated with
the Heegaard diagram  $R_{n}=(\Sig,\alphas,\betas_{n};u,v)$ and the relative 
$\SpinC$ class $\spinc$ is denoted by $\ov{\mathrm{CFK}}(K_{n},\spinc)$, while 
the complex associated with the Heegaard diagram  
$R_{n+m}=(\Sig,\alphas,\betas_{n+m};u,v)$ and the relative $\SpinC$ classes 
$\spinc$ and $\spinc+m$ is denoted by
$$\ov{\mathrm{CFK}}(K_{m+n},\spinc)\oplus \ov{\mathrm{CFK}}(K_{m+n},\spinc+m).$$

Let $\Theta_f$ denote the top generator of the Heegaard Floer homology 
group associated with $(\Sig,\betas_{n+m},\betas;u,v)$. 
Consider the holomorphic triangle map
\begin{equation}\label{eq:f}
\begin{split}
&f^\spinc:\ov{\mathrm{CFK}}(K_{m+n},\spinc)\oplus 
\ov{\mathrm{CFK}}(K_{m+n},\spinc+m)\ra \ov{\mathrm{CF}}(Y)\\
&f^\spinc(\x):=\sum_{\z\in \Ta\cap \mathbb{T}_{\beta}}
\sum_{\substack{\Delta\in \pi_2^0(\x,\Theta_f,\z;u,v)%,\\  \mu(\Delta)=0,\
 %n_u(\Delta)=n_v(\Delta)=0
 }}\#\big(\ov{\Mod}(\Delta)\big).\z.\\
\end{split}
\end{equation}
The diagram $(\Sig,\alphas,\betas_{n},\betas_{n+m};u,v)$ determines a 
 cobordism  from $Y_{n}(K)\coprod L $
 to $Y_{n+m}(K)$, where $L=L(m,1)\#(\#^{g-1}S^1\times S^2)$.
 The intersection point $q$ determines a canonical 
 $\SpinC$ class $\spinc_q\in \SpinC(L)$ in the sense of Definition 3.2 of \cite{OS-surgery}.
 Let $\Theta_g$ denote the top generator of 
 $ \ov{\CFT}(\Sig,\betas_n,\betas_{n+m};u,v)$
 which corresponds to $\spinc_q$, or equivalently to the intersection point $q$.
 Define
\begin{displaymath}
\begin{split}
&g^\spinc:\ov\CFKT(K_n,\spinc)\lra \ov\CFKT(K_{n+m})\\
&g^\spinc(\x):=\sum_{\z\in \Ta\cap \mathbb{T}_{\beta_{n+m}}}
\sum_{\substack{\Delta\in \pi_2^0(\x,\Theta_g,\z;u,v) }}
\#\big(\ov{\Mod}(\Delta)\big).\z.\\
\end{split}
\end{displaymath}
Following Section 8 of \cite{AE} (Lemma 8.2 and 
the discussion after that), if $\spinc(\x)=\spinc$
\begin{displaymath}
g^\spinc(\x)\in \ov{\mathrm{CFK}}(K_{n+m},\spinc)\oplus
\ov{\mathrm{CFK}}(K_{n+m},m+\spinc).
\end{displaymath}

Finally, the top generator $\Theta_h\in\ov{\CFT}(\Sig,\betas,\betas_n;u,v)$ and the 
Heegaard triple $(\Sig,\alphas,\betas,\betas_{n};u,v)$ determine the map $h^\spinc$
on $\ov\CFT(Y)$, which is defined by
\begin{displaymath}
\begin{split}
&h^\spinc(\x):=\sum_{\substack{\z\in \Ta\cap \mathbb{T}_{\beta_{n}}\\ \spinc(\z)=\spinc}}
\sum_{\substack{\Delta\in \pi_2^0(\x,\Theta_h,\z;u,v)
%\\ n_w(\Delta)+\spinc(\x)\equiv\spinc\ (\mathrm{mod}\ m)
}}\#\big(\ov{\Mod}(\Delta)\big).\z.
\end{split}
\end{displaymath}
We thus arrive at the triangle of chain maps
\begin{equation}\label{eq:hol-triangle}
\begin{diagram}
\ov{\mathrm{CF}}(Y)&&\rTo{h^\spinc=h^\spinc_n}&&
\ov{\mathrm{CFK}}(K_{n},\spinc)\\
&\luTo{f^\spinc=f^\spinc_n}&&\ldTo{g^\spinc=g^\spinc_n}&\\
&&\ov{\mathrm{CFK}}(K_{m+n},\spinc)\oplus \ov{\mathrm{CFK}}(K_{m+n},\spinc+m)&&
\end{diagram}
\end{equation}

\subsection{Exactness of triangle}
Let $M(f^\spinc_n)$ denote the mapping cone of $f^\spinc=f_n^\spinc$. 

\begin{thm}\label{thm:mapping-cone}
If $m$ is sufficiently  large  there is a map 
$$H_{h_n}^\spinc:\ov\CFKT(K_n,\spinc)\lra \ov\CFT(Y)$$
which satisfies $d\circ H_{h_n}^\spinc+H_{h_n}^\spinc\circ d=f_n^\spinc\circ g_n^\spinc$,
such that the chain map 
\begin{displaymath}
\begin{split}
&\imath_n^\spinc:\ov\CFKT(K_n,\spinc)\lra M(f_n^\spinc)\\
&\imath_n^\spinc(\x):=(g_n^\spinc(\x),H_{h_n}^\spinc(\x)),\ \ \ \forall\ \x
\in\ov\CFKT(K_n,\spinc),
\end{split}
\end{displaymath}
is a quasi-isomorphism.
\end{thm}
\begin{proof}
The proof is almost identical to the 
proof used in Section 8 from \cite{AE}. We  outline the proof to set up the notation.
Let us first define
\begin{displaymath}
\begin{split}
&H_f^\spinc:\ov\CFT(Y)\ra \ov\CFKT(K_{n+m},\spinc)\oplus \ov\CFKT(K_{n+m},m+\spinc)\\
&H_{f}^\spinc(\x):=
\sum_{\substack{\y\in\Ta\cap\mathbb{T}_{\beta_{n+m}}\\ \spinc(\y)-\spinc
\equiv 0\ (\mathrm{mod}\ m)}}
\sum_{\substack{\square\in\pi_2^{-1}(\x,\Theta_h,\Theta_g,\y;u,v)}}
\#\big(\Mod(\square)\big).\y.
\end{split}
\end{displaymath}
%\end{equation}
 The condition $\spinc(\y)-\spinc\equiv (\mathrm{mod}\ m)$ implies 
$\spinc(\y)\in\{\spinc,\spinc+m\}$ since $m$ is large.
Considering all possible boundary degenerations of the one-dimensional moduli
space corresponding to a square class $\square\in\pi_2^0(\x,\Theta_h,\Theta_g,\y;u,v)$
we find  
\begin{equation}\label{eq:H-f}
d\circ H_f^\spinc+H_f^\spinc\circ d=h^\spinc\circ g^\spinc.
\end{equation}
For (\ref{eq:H-f}) note that the holomorphic triangles 
$\Delta\in\pi_2^0(\Theta_h,\Theta_g,\Theta;u,v)$ with $\Theta\in\mathbb{T}_{\beta_{n+m}}
\cap\Tb$ come in canceling pairs, where 
the difference between the coefficients of every canceling pair at a marking $s$ (placed
on the left-hand-side of $\beta_g$ and close to $u$) is always a multiple of $m$.
Similarly, define
%\begin{equation}\label{eq:H-g-def}
\begin{displaymath}
\begin{split}
&H_{g}^\spinc:\ov\CFKT(K_{n+m},\spinc)\oplus \ov\CFKT(K_{n+m},m+\spinc)
\lra \ov\CFT(Y),\\
&H_{g}^\spinc(\x):=
\sum_{\y\in\Ta\cap\mathbb{T}_{\beta_{n}}}
\sum_{\substack{\square\in\pi_2^{-1}(\x,\Theta_f,\Theta_h,\y;u,v)\\
n_s(\square)\equiv 0\ (\mathrm{mod}\ m)}}\#\big(\Mod(\square)\big).\y.\\
\end{split}
\end{displaymath}
%\end{equation}
Since the contributing holomorphic 
 triangles corresponding to the Heegaard triple 
 $(\Sig,\betas_{n+m},\betas,\betas_n;u,v)$
 and the closed top generators $\Theta_f,\Theta_h$ come in canceling pairs, 
\begin{equation}\label{eq:H-g}
d\circ H_{g}^\spinc+H_{g}^\spinc=h^\spinc\circ f^\spinc.
\end{equation}
Finally, define the homotopy map $H_h^\spinc$ by
\begin{displaymath}
\begin{split}
&H_{h}^\spinc:\ov{\mathrm{CFK}}(K_{n},\spinc)\lra\ov{\mathrm{CF}}(Y) \\
&H_{h}^\spinc(\x)=\sum_{\y\in\Ta\cap\mathbb{T}_{\beta}}\ 
\sum_{\substack{\square\in\pi_2^{-1}(\x,\Theta_g,\Theta_f,\y;u,v)}}
\#\big(\Mod(\square)\big).\y.
\end{split}
\end{displaymath}
 Employ the same argument again to show that
 \begin{equation}\label{eq:H-h}
 d\circ H_h^\spinc+H_h^\spinc\circ d=f^\spinc\circ g^\spinc. 
\end{equation}  

We next introduce the pentagon maps. Let
 $\betas_n'$ denote a $g$-tuple of simple closed curves which are small Hamiltonian 
 isotopes of the curves in $\betas_n$. Choosing the Hamiltonian isotopy sufficiently small
 allows us to assume that the chain complex associated with 
 $(\Sig,\alphas,\betas_n';u,v)$ and the $\SpinC$ class $\spinc$ may be identified 
 with $\ov{\CFKT}(K_n,\spinc)$, since the  intersection points and the corresponding 
 moduli spaces connecting them change continuously  by slight Hamiltonian 
 perturbation of the  Lagrangian sub-manifolds. There is a top generator corresponding to
 $(\betas,\betas_n')$ which is in correspondence with $\Theta_h$. We denote this 
 generator by $\Theta_h'$.
 Define
 %\begin{equation}\label{eq:P-f-def}
\begin{displaymath}
\begin{split}
 &P_{f}^\spinc:\ov{\CFKT}(K_n,\spinc)\lra \ov{\CFKT}(K_n,\spinc),\\
 &P_{f}^\spinc(\x):=\sum_{\y\in\Ta\cap\mathbb{T}_{\beta_n}}
 \sum_{\substack{\pentagon\in\pi_2^{-2}(\x,\Theta_g,\Theta_f,\Theta_h',\y;u,v)\\
  n_s(\pentagon)\equiv 0\ (\mathrm{mod}\ m)}}
  \#\big(\Mod(\pentagon)\big).\y,\\
 \end{split} 
\end{displaymath} 
 %\end{equation}
 Consider different boundary degenerations of the $1$-dimensional moduli
 space associated with a pentagon of Maslov index $-1$, and note that
 \begin{itemize}
 \item There is a unique contributing square classes 
 $\square\in\pi_2^{-1}(\Theta_g,\Theta_f,\Theta_h',\Theta)$ 
 of index $-1$ which corresponds to  the %punctured Heegaard 
 quadruple
 $(\Sig,\betas_n,\betas_{n+m},\betas,\betas_n';u,v)$. Moreover, 
 $\Theta=\Theta_{n}$
 is the top generator for the %Heegaard 
 diagram $(\Sig,\betas_n,\betas_n';u,v)$.
 \item The contributing triangle classes 
 $$\Delta\in\pi_2^0(\Theta_g,\Theta_f,\Theta;u,v)\ \ \text{and}\ \  
 \Delta'\in\pi_2^0(\Theta_f,\Theta_h',\Theta;u,v)$$ 
corresponding to the triples 
$(\Sig,\betas_n,\betas_{n+m},\betas;u,v)$ and 
$ (\Sig,\betas_{n+m},\betas,\betas_n';u,v)$ come 
 in canceling pairs. 
 \end{itemize}
 These observations, combined with our earlier arguments, imply that
 \begin{equation}\label{eq:P-f}
 \begin{split}
 &d\circ P_{f}^\spinc+P_{f}^\spinc \circ d+J_{f}^\spinc
 =h^\spinc\circ H_{h}^\spinc- H_{g}^\spinc\circ g^\spinc,\ \ \text{where}\\
& J_{f}^\spinc(\x)=\sum_{\y\in\Ta\cap\mathbb{T}_{\beta_n'}}
 \sum_{\substack{\Delta\in\pi_2^0(\x,\Theta_n,\y;u,v)}}\#\big(\Mod(\Delta)\big).\y.
 \end{split}
 \end{equation}
 
Consider the $5$-tuple
$(\Sig,\alphas,\betas,\betas_n,\betas_{n+m},\betas';u,v,z)$,
where $\betas'=\{\beta_1',...,\beta_g'\}$ 
is a set of $g$ simple closed curves which are obtained from 
$\betas$ by a small Hamiltonian isotopy. Thus $\beta_i$ and $\beta_i'$ intersect
each other is a pair of canceling intersection points. We assume that the small area bounded 
between the two curves $\beta_g$ and $\beta_g'$ is formed as a union of two bigons; a
small bigon which is a subset of the connected component of 
$\Sig^\circ=\Sig-\alphas-\betas-\betas_n-\betas_{n+m}$ 
which contains the marking $v$ and a 
long and thin bigon which is stretched along $\beta_g$. We assume that the 
marking $z$ is chosen in the intersection of the second bigon with 
the connected component in $\Sig^\circ$ which corresponds to $u$. If the Hamiltonian 
perturbation is sufficiently small  the chain complex $\ov{\CFT}(\Sig,\alphas,\betas';u)$
may be identified with $\ov{\CFT}(Y)$. Define  
%\begin{equation}\label{eq:P-g-def}
\begin{displaymath}
\begin{split}
&P_{g}^\spinc:\ov{\CFT}(Y)\lra \ov\CFT(Y)\\
&P_{g}^\spinc(\x):=\sum_{\y\in\Ta\cap\mathbb{T}_{\beta'}}
\sum_{\substack{\pentagon\in\pi_2^{-2}(\x,\Theta_h,\Theta_g,\Theta_f',\y;u,v)\\
n_z(\pentagon)+\spinc(\x)\equiv \spinc\ (\mathrm{mod}\ m)}}
\#\big(\Mod(\pentagon)\big).\y.
\end{split}
\end{displaymath}
%\end{equation}
Five types of the ten possible degenerations in the boundary of the 
$1$-dimensional moduli space associated with  a pentagon class 
$$\pentagon\in\pi_2^{-1}(\x,\Theta_h,\Theta_g^q,\Theta_f',\y;u,v)\ \ 
\text{with}\ \  n_z(\pentagon)+\spinc(\x)\equiv \spinc\ (\mathrm{mod}\ m),$$
corresponding to a degeneration to a bigon and a pentagon, contribute to the coefficient 
of $\y$ in $(d\circ P_{g}^\spinc+P_{g}^\spinc\circ d)(\x)$.
The remaining five types correspond to the 
degenerations of $\pentagon$ into a square and a 
triangle. The choice of the markings implies that two of these degeneration types 
contribute to the coefficient of $\y$ in
$(f^\spinc\circ H_{f}^\spinc-H_{h}^\spinc\circ h^\spinc)(\x)$.
There is a unique contributing square class, corresponding to 
$(\Sig,\betas,\betas_n,\betas_{n+m},\betas';u,v)$ and the intersection points 
$\Theta_h,\Theta_g,\Theta_f',\Theta_\infty$, where $\Theta_\infty$
denotes the top generator for $(\Sig,\betas,\betas';u,v)$. Moreover, the triangles which 
contribute in $\pi_2(\Theta_h,\Theta_g,\Theta_f)$ and 
$\pi_2(\Theta_g,\Theta_f',\Theta_h')$ come in canceling pairs. Thus
 \begin{equation}\label{eq:P-g}
 \begin{split}
& d\circ P_{g}^\spinc+P_{g}^\spinc\circ d+ J_g^\spinc=
 f^\spinc\circ H_{f}^\spinc-H_{h}^\spinc\circ h^\spinc,\ \ \text{where}\\
&J_g^\spinc(\x)=\sum_{\y\in\Ta\cap\mathbb{T}_{\beta'}}
\sum_{\substack{\Delta\in\pi_2^0(\x,\Theta_{\infty},\y;u,v)}}
\#\big(\Mod(\Delta)\big).\y
 \end{split} 
 \end{equation}

Let $\betas_{n+m}'$ denote a $g$-tuple of simple closed curves which are small Hamiltonian 
 isotopes of the curves in $\betas_{n+m}$. Again, we 
 assume that the chain complex associated with 
 $(\Sig,\alphas,\betas_{n+m}';u,v)$ and the $\SpinC$ classes $\spinc,\spinc+m$ is
 identified  with $\ov{\CFKT}(K_{n+m},\spinc)\oplus \ov{\CFKT}(K_{n+m},\spinc+m)$. 
 There is a top generator $\Theta_g'$ for 
 $(\betas_n,\betas_{n+m}')$ which is in correspondence with $\Theta_g$. 
 Define
\begin{displaymath}
\begin{split}
&P_{h}^\spinc:\bigoplus_{\spinct\in\{\spinc,\spinc+m\}}
\ov{\CFKT}(K_{n+m},\spinct)
\lra \bigoplus_{\spinct\in\{\spinc,\spinc+m\}}
\ov{\CFKT}(K_{n+m},\spinct)\\
&P_{h}^\spinc(\x):=\sum_{\y\in\Ta\cap\mathbb{T}_{\beta_{n+m}'}}
\sum_{\substack{\pentagon\in\pi_2^{-2}(\x,\Theta_f,\Theta_h,\Theta_g',\y;u,v)\\
n_z(\pentagon)\equiv 0\ (\mathrm{mod}\ m)}}
\#\big(\Mod(\pentagon)\big).\y.
\end{split}
\end{displaymath}
A similar argument implies that
 \begin{equation}\label{eq:P-h}
 \begin{split}
& d\circ P_{h}^\spinc+P_{h}^\spinc\circ d+ J_h^\spinc=
 g^\spinc\circ H_{g}^\spinc-H_{f}^\spinc\circ f^\spinc,\ \ \ \text{where}\\
&J_h^\spinc(\x)=\sum_{\y\in\Ta\cap\mathbb{T}_{\beta_{n+m}'}}
\sum_{\substack{\Delta\in\pi_2^0(\x,\Theta_{n+m},\y;u,v)}}
\#\big(\Mod(\Delta)\big).\y,
 \end{split} 
 \end{equation}
and $\Theta_{n+m}$ is the top generator of $(\Sig,\betas_{n+m},\betas_{n+m}';u,v)$.
Since $J_f^\spinc,J_g^\spinc,J_h^\spinc$ are  quasi-isomorphisms, 
Lemma 3.3 from \cite{AE} completes the proof.
\end{proof} 

Choose the markings $s$ and $t$ on the Heegaard diagram so that for each one of the pairs
 $(z,s)$ and $(v,t)$ there is an arc connecting them on the Heegaard surface which cuts 
 $\beta_g$ in a single  
 transverse point and stays disjoint from all other curves in 
 $\alphas\cup\betas\cup\betas'\cup\betas_n\cup\betas_{n+m}$, see Figure~\ref{fig:pentagon}.
   Consider the chain map
\begin{displaymath}
\begin{split}
&\taub\circ \ovl{f}^\spinc:\ov\CFKT(K_{n+m},\spinc)\oplus\ov\CFKT(K_{n+m},\spinc+m)
\lra \ov\CFT(Y),\\
&\ovl{f}^\spinc(\x)=\sum_{\y\in\Ta\cap\Tb}\sum_{\substack{\Delta\in
\pi_2^0(\x,\y;s,t)}}\#(\Mod(\Delta)).\y,
\end{split}
\end{displaymath}
where $\taub:\ov\CFT(\Sig,\alphas,\betas;s)\ra\ov\CFT(\Sig,\alphas,\betas;u)$
is the chain homotopy equivalence given by the Heegaard moves which change 
$(\Sig,\alphas,\betas;s)$ to $(\Sig,\alphas,\betas;u)$.

\begin{lem}\label{lem:f-chain-homotopy}
The chain maps $f^\spinc$ and $\taub\circ \ovl{f}^\spinc$ are chain homotopic.
\end{lem}
\begin{proof}
Note that the aforementioned Heegaard 
moves consist of $2g-2$ handle slides (composed with isotopies) 
on $\betas$, supported away from the markings $s,t$. Denote the corresponding
$g$-tuples of curves by $\betas^0=\betas,\betas^1,...,\betas^{2g-2}$, where 
$\ov\CFT(\alphas,\betas^{2g-2};s)$ may be identified with  $\ov\CFT(\Sig,\alphas,\betas;u)$.\\
The Heegaard triple $(\Sig,\alphas,\betas^{i-1},\betas^i;s)$ and the top generator
$\Theta^i$ of $(\Sig,\betas^{i-1},\betas^i;s,t)$ determine a chain map
$$\taub^i:\ov\CFT(\alphas,\betas^{i-1};s)\lra \ov\CFT(\alphas,\betas^i;s).$$
The Heegaard triple $(\Sig,\alphas,\betas_{n+m},\betas^i;s,t)$ together
with the top generator $\Theta_f^{i}$ of $(\Sig,\betas_{n+m},\betas^{i};s,t)$
 determines a chain map
 $$f^i:\ov\CFT(\alphas,\betas_{n+m};s,t)\lra \ov\CFT(\alphas,\betas^i;s).$$
Finally, the Heegaard quadruple $(\Sig,\alphas,\betas_{n+m},\betas^{i-1},\betas^i;s,t)$
 together with $\Theta_f^{i-1}$ and $\Theta^i$, determines a homomorphism
\begin{displaymath}
\begin{split}
H^i:\ov\CFT(\Sig,\alphas,\betas_{n+m};s,t)\lra \ov\CFT(\Sig,\alphas,\betas^i;s).
\end{split}
\end{displaymath}

Considering different boundary degenerations of the one-dimensional moduli space 
associated with a square class of index $0$ we find
\begin{equation}\label{eq:H-i}
d\circ H^i+H^i\circ d=f^i+\taub^i\circ f^{i-1},\ \ \ \ i=1,...,2g-2.
\end{equation}  
Let us define
\begin{displaymath}
H=H^{2g-2}+\taub^{2g-2}\circ H^{2g-3}+\taub^{2g-2}\circ \taub^{2g-3} H^{2g-4}+
\dots + (\taub^{2g-2}\circ \dots \taub^2)\circ H^1.
\end{displaymath}
Using (\ref{eq:H-i}) we find
\begin{displaymath}
d\circ H+H\circ d=f^{2g-2}+\taub\circ f^0,\ \ \ \text{where}\ 
\taub=\taub^{2g-2}\circ\dots\taub^1.
\end{displaymath}
Restricting the above equation to 
$\ov\CFKT(K_{n+m},\spinc)\oplus\ov\CFKT(K_{n+m},\spinc+m)$ we are done, once we note 
that $f^\spinc$ is the restriction of $f^{2g-2}$ and $\ovl{f}^\spinc$ is the restriction 
of $f^0$.
\end{proof}

%\newpage
\section{The homomorphisms in the surgery triangle}\label{sec:naturality}
Consider the triply punctured Heegaard $5$-tuple 
$(\Sig,\alphas,\betas_0,\betas_1,\betas_m,\betas;u,v,w)$,
as before, and assume that the local picture around the  curves 
$\lambda_0,\lambda_1,\lambda_m,\lambda_\infty$ is the one illustrated in 
Figure~\ref{fig:pentagon}.
The top generators $\Theta_{0,1}$, $\Theta_{g_1}$ and $\Theta_{f_1}$ 
of the Heegaard diagrams $(\Sig,\betas_0,\betas_1;u,v,w)$,
$(\Sig,\betas_1,\betas_m;u,v,w)$ and 
$(\Sig,\betas_m,\betas;u,v,w)$ (respectively) determine 
the holomorphic pentagon map
\begin{displaymath}
\begin{split}
&P^\spinc:\ov\CFKT(K_0,\spinc)\lra \ov\CFT(Y)\\
&P^\spinc(\x):=\sum_{\y\in\Ta\cap\Tb}\sum_{\substack{
\pentagon\in\pi_2^{-2}(\x,\Theta_{0,1},\Theta_{g_1},\Theta_{f_1},\y;u,v,w)}}
\#\left(\Mod(\pentagon)\right).\y
\end{split}
\end{displaymath} 
Every pentagon class 
$\pentagon\in\pi_2^{-1}(\x,\Theta_{0,1},\Theta_{g_1},\Theta_{f_1},\y;u,v,w)$
corresponds to a $1$-dimensional moduli space with boundary. The boundary points 
are in correspondence with  the degeneration of the domain of 
$\pentagon$ into two parts. Since the generators  $\Theta_{0,1},\Theta_{g_1}$ and 
$\Theta_{f_1}$ are closed, the degenerations into a bi-gon and a pentagon 
correspond to the the coefficient of $\y$ in $(d\circ P^\spinc+P^\spin\circ d)(\x)$. 
The remaining  degenerations are the  degenerations 
$\pentagon=\square\star \Delta$ 
to a triangle $\Delta$ with Maslov index $0$ and a square $\square$
with Maslov index $-1$ which miss $u,v$ and $w$. The possibilities are
\begin{itemize}
\item[(1)] %Degenerations with 
$\square\in\pi_2(\z,\Theta_{g_1},\Theta_{f_1},\y)$ and
$\Delta\in\pi_2(\x,\Theta_{0,1},\z)$,
\item[(2)] %Degenerations with 
$\square\in\pi_2(\x,\Theta_{0,1},\Theta_{g_1},\z)$ and
$\Delta\in\pi_2(\z,\Theta_{f_1},\y)$,
\item[(3)] %Degenerations with 
$\square\in\pi_2(\x,\Theta_{0,1},\Theta,\y)$ and
$\Delta\in\pi_2(\Theta_{g_1},\Theta_{f_1},\Theta)$,
\item[(4)] %Degenerations with 
$\square\in\pi_2(\x,\Theta,\Theta_{f_1},\y)$ and
$\Delta\in\pi_2(\Theta_{0,1},\Theta_{g_1},\Theta)$,
\item[(5)] %Degenerations with 
$\square\in\pi_2(\Theta_{0,1},\Theta_{g_1},\Theta_{f_1},\Theta)$ and 
$\Delta\in\pi_2(\x,\Theta,\y)$.
\end{itemize}
First type degenerations  correspond to the coefficient of $\y$ in 
$(H_{h_1}^\spinc\circ \Ff_\infty^\spinc)(\x)$, where 
$$\Ff_\infty^\spinc:\ov\CFKT(K_0,\spinc)\lra \ov\CFKT(K_1,\spinc)$$ 
is a chain map so that the induced map  
$\Ff_\infty^\spinc:\Hbb_0(K,\spinc)\ra \Hbb_1(K,\spinc)$ happens to be the 
homomorphism which appears in the splicing formula of \cite{Ef-splicing}. 
Degenerations of 
type $2$ correspond to the coefficient of $\y$ in $f_1^\spinc\circ H^\spinc(\x)$, 
where $H^\spinc$ is defined by
\begin{displaymath}
\begin{split}
&H^\spinc:\ov\CFKT(K_0,\spinc)\lra \ov\CFKT(K_m,\spinc)\oplus
\ov\CFKT(K_m,\spinc+m-1)\\
&H^\spinc(\x)=\sum_{\z\in\Ta\cap\mathbb{T}_{\beta_m}}\sum_{\substack{
\square\in\pi_2^{-1}(\x,\Theta_{0,1},\Theta_{g_1},\z;u,v,w)}}\#\left(\Mod(\square)\right).\z.
\end{split}
\end{displaymath}
In a degeneration of type $3$, the contributing triangle classes $\Delta$ come 
in canceling pairs. The total (signed) count of such degenerations is thus trivial. 
Furthermore, there are no holomorphic representatives for the square classes 
which appear in the boundary degenerations of type $5$, i.e. we may assume that 
there are no such degenerations.
In a degeneration of type $4$, the moduli space corresponding to $\Delta$ is trivial 
unless $\Theta=\Theta_{g_0}$ and $\Delta$ corresponds to the union of small 
triangles connecting $\Theta_{0,1},\Theta_{g_1}$ and $\Theta_{g_0}$. In this 
latter case the signed contribution of such triangles is $1$. The signed count of 
such boundary degenerations is thus equal to 
\begin{displaymath}
\sum_{\y\in\Ta\cap\Tb}\sum_{\substack{\square\in\pi_2^{-1}
(\x,\Theta_{g_0},\Theta_{f_1},\y;u,v,w)}}
\#\left(\Mod(\square)\right).\y=H_{h_0}^\spinc(\x).
\end{displaymath}
Summarizing the above observations we arrive at the following.

\begin{lem}\label{lem:pentagon}
With the above notation fixed
\begin{equation}\label{eq:pentagon}
d\circ P^\spinc+P^\spinc\circ d+H_{h_0}^\spinc=f_1^\spinc\circ H^\spinc
-H_{h_1}^\spinc\circ \Ff_\infty^\spinc.
\end{equation}
\end{lem} 

\begin{figure}[ht]
\def\svgwidth{10cm}
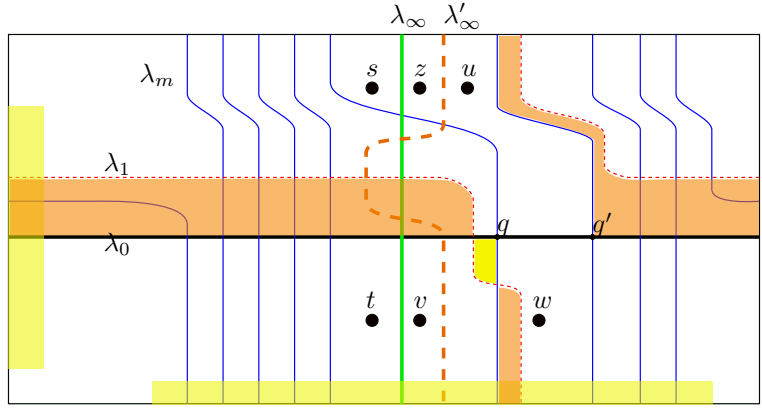
\caption{The arrangement of the curves on the Heegaard surface. Other curves and handles 
appear on the shaded yellow area. }\label{fig:pentagon}
\end{figure}  

Next, we analyse the map $H^\spinc$ via degenerations of holomorphic squares.
For a square class 
$\square\in\pi_2^0(\x,\Theta_{0,1},\Theta_{g_1},\y;u,v,w)$
the moduli space $\Mod(\square)$ is one dimensional, and has $6$ types of boundary 
ends, corresponding to the degenerations of the domain. Since $\Theta_{0,1}$ and 
$\Theta_{g_1}$ are closed, the $4$ types of degenerations of the square class
to a square and a bi-gon correspond to the coefficient of $\y$ in 
$(d\circ H^\spinc+H^\spinc\circ d)(\x)$. The remaining boundary ends correspond 
to a degeneration of $\square$ to a pair of triangle classes. The degenerations 
$\square=\Delta'\star\Delta$ with $\Delta\in\pi_2(\x,\Theta_{0,1},\z)$ and 
$\Delta'\in\pi_2(\z,\Theta_{g_1},\y)$ correspond to the coefficient of 
$\y$ in $(g_1^\spinc\circ \Ff_\infty^\spinc)(\x)$.\\

 The more tricky and interesting part is the contribution of the boundary ends which 
 correspond to the degenerations of the form 
 $\square=\Delta'\star\Delta$ with 
 $$\Delta'\in\pi_2^0(\Theta_{0,1},\Theta_{g_1},\Theta;u,v,w)\ \ \text{and}\ \ 
\Delta\in\pi_2^0(\x,\Theta,\y;u,v,w). $$
 There are precisely two generators $\Theta$ such that there is a corresponding 
 $\Delta'=\Delta_\Theta$ associated with them such that $\Mod(\Delta_\Theta)$ 
 is non-empty. One of these classes corresponds to $\Theta=\Theta_{g_0}$ and 
 the other one corresponds to the generator $\Theta_{g_0}'$ which is 
 obtained from $\Theta_{g_0}$ by changing $q$ to the intersection point 
 $q'\in \lambda_0\cap\lambda_m$ which is next to $q$ (see Figure~\ref{fig:pentagon}). 
 The total  contribution of such boundary ends is thus 
 \begin{displaymath}
 \begin{split}
 &
 \sum_{\substack {\y\in\Ta\cap\mathbb{T}_{\beta_m}\\
 \Delta\in\pi_2^0(\x,\Theta_{g_0},\y;u,v,w)}}\#\left(\Mod(\Delta)\right).\y+
\sum_{\substack{\y\in\Ta\cap\mathbb{T}_{\beta_m}\\
 \Delta\in\pi_2^0(\x,\Theta_{g_0}',\y;u,v,w)}}\#\left(\Mod(\Delta)\right).\y.
 \end{split}
\end{displaymath}

Let $g_0^\spinc(\x)=(g_{0,1}^\spinc(\x),g_{0,2}^\spinc(\x))$ denotes the decomposition 
of $g_0^\spinc$ in $$\ov\CFKT(K_m,\spinc)\oplus\ov\CFKT(K_m,\spinc+m).$$ Then the 
above sum is equal to $g_{0,1}^\spinc(\x)+G^\spinc(\x)$, where  
$$G^\spinc:\ov\CFKT(K_0,\spinc)\lra \ov\CFKT(K_m,\spinc+m-1)$$
is defined by the second sum above.
In  Section~\ref{sec:splicing} we show that for sufficiently large  
$m$ and an appropriate Heegaard $5$-tuple we may assume 
that there is an embedding 
$$J^\spinc:\ov\CFKT(K_m,\spinc+m)\lra \ov\CFKT(K_m,\spinc+m-1)$$
such that $G^\spinc(\x)=J^\spinc(g_{0,2}^\spinc(\x))$. Set
\begin{displaymath}
\begin{split}
&G_\infty^\spinc:\bigoplus_{\spinct\in\{\spinc,\spinc+m\}}
\ov\CFKT(K_m,\spinct)\lra \bigoplus_{\spinct\in\{\spinc,\spinc+m-1\}}
\ov\CFKT(K_m,\spinct),\\
&G_\infty^\spinc(\x_1,\x_2):=(\x_1,J^\spinc(\x_2)).
\end{split}
\end{displaymath}
The above observations imply the following.
\begin{lem}\label{lem:square}
With the above notation fixed we have
\begin{equation}\label{eq:square}
d\circ H^\spinc+H^\spinc\circ d=g_1^\spinc\circ 
\Ff_\infty^\spinc-G_\infty^\spinc\circ g_0^\spinc.
\end{equation}
\end{lem}

Let us now consider the Heegaard $5$-tuple
$\Hcal=(\Sig,\alphas,\betas_1,\betas_m,\betas,\betas';u,v,z)$
where the curves in $\betas'=\{\beta_1',...,\beta_g'\}$ are small Hamiltonian isotopes 
of the corresponding curves in $\betas$. Moreover, we assume that the intersection pattern
between $\lambda_1,\lambda_m,\lambda_\infty=\beta_g$ and $\lambda_\infty'=\beta_g'$
and the location of $u,v$ and $z$ follows the pattern illustrated in 
Figure~\ref{fig:pentagon}. Using the punctured Heegaard diagram 
$(\Sig,\alphas,\betas';u,v,z)$ we may form the chain complex associated 
with $K$ and $\spinc\in\RelSpinC(Y,K)$. We will thus denote this latter 
chain complex by $\ov\CFKT(K,\spinc)$. Associated 
with the Heegaard diagram  $(\Sig,\betas,\betas';u,v,z)$ there is a {\emph{top generator}}
which may be denoted by $\Theta_\infty'$. Unlike most of such situations $\Theta_\infty'$ 
is not closed and 
$d(\Theta_\infty')=\Theta_\infty$  is the generator which is obtained from $\Theta_\infty'$ 
by changing the choice of intersection point in $\lambda_\infty\cap\lambda_\infty'$.
By construction, $\Theta_\infty$ is closed.
The diagram $\Hcal$ defines a pentagon map
\begin{displaymath}
\begin{split}
&Q^\spinc:\ov\CFKT(K_1,\spinc)\lra \ov\CFKT(K,\spinc),\\
&Q^\spinc(\x)=\sum_{\substack{\y\in\Ta\cap\mathbb{T}_{\beta'},\ \spinc(\y)=\spinc\\ 
\pentagon\in\pi_2^{-2}(\x,\Theta_{g_1},\Theta_{f_1},\Theta_\infty,\y;u,v,z)
}}
\#\left(\Mod(\pentagon)\right).\y.
\end{split}
\end{displaymath}
For $\pentagon\in\pi_2^{-1}(\x,\Theta_{g_1},\Theta_{f_1},\Theta_\infty,\y;u,v,z)$
consider the ends of the $1$-dimensional moduli space $\Mod(\pentagon)$,
which  correspond to the degenerations of the 
pentagon either to a bi-gon and a pentagon, or to a triangle and a square. There are 
five types of degenerations of each one of these two forms. Since $\Theta_{g_1},
\Theta_{f_1}$ and $\Theta_\infty$ are closed, the first five types of degenerations 
correspond to the coefficient of $\y$ in $(d\circ Q^\spinc+Q^\spinc\circ d)(\x)$. 
Other degeneration are of the  form $\pentagon=\square\star \Delta$ of
 one of the following $5$ types:
\begin{itemize}
\item[(1)] %Degenerations with 
$\square\in\pi_2(\z,\Theta_{f_1},\Theta_{\infty},\y)$ and
$\Delta\in\pi_2(\x,\Theta_{g_1},\z)$,
\item[(2)] %Degenerations with 
$\square\in\pi_2(\x,\Theta_{g_1},\Theta_{f_1},\z)$ and
$\Delta\in\pi_2(\z,\Theta_{\infty},\y)$,
\item[(3)] %Degenerations with 
$\square\in\pi_2(\x,\Theta_{g_1},\Theta,\y)$ and 
$\Delta\in\pi_2(\Theta_{f_1},\Theta_{\infty},\Theta)$,
\item[(4)] %Degenerations with 
$\square\in\pi_2(\x,\Theta,\Theta_{\infty},\y)$ and 
$\Delta\in\pi_2(\Theta_{g_1},\Theta_{f_1},\Theta)$,
\item[(5)] %Degenerations with 
$\square\in\pi_2(\Theta_{g_1},\Theta_{f_1},\Theta_{\infty},\Theta)$  and
$\Delta\in\pi_2(\x,\Theta,\y)$.
\end{itemize}
The first type in the above list corresponds to the coefficient of $\y$ in the expression
$(I^\spinc\circ g_1^\spinc)(\x)$, where $I^\spinc$ is defined by
\begin{displaymath}
I^\spinc(\z)=\sum_{\substack{\y\in\Ta\cap\mathbb{T}_{\beta'},\ \spinc(\y)=\spinc\\
\square\in \pi_2^{-1}(\z,\Theta_{f_1},\Theta_{\infty},\y;u,v,z)}}
\#\left(\Mod(\square)\right).\y.
\end{displaymath}
The second type in the above list corresponds to the coefficient of $\y$ in 
$(X^\spinc\circ H_{h_1}^\spinc)(\x)$, where 
$X^\spinc$ is defined by 
\begin{displaymath}
X^\spinc(\z)=\sum_{\substack{\y\in\Ta\cap\mathbb{T}_{\beta'},\ \spinc(\y)=\spinc\\
\Delta\in \pi_2^{0}(\z,\Theta_{\infty},\y;u,v,z)}}
\#\left(\Mod(\Delta)\right).\y.
\end{displaymath}
Considering the local multiplicities around $\lambda_\infty\cap\lambda_\infty'$
one may conclude that there are no triangle classes 
$\Delta\in\pi_2^0(\z,\Theta_\infty,\y;u,v,z)$ with positive domain. 
In particular,  $X^\spinc$ is trivial.
There are no triangle classes which contribute in the degenerations of the type (3).
The triangles which contribute in degenerations of type (4) come in canceling 
pairs. Thus the total number of boundary ends corresponding 
to degenerations of types (3) and (4) is zero.
There is a unique square class 
$$\square\in\pi_2^{-1}(\Theta_{g_1},\Theta_{f_1},\Theta_{\infty},\Theta;u,v,z)$$
with non-trivial contribution to the degenerations of type (5). For this square 
class $\Theta\in\mathbb{T}_{\beta_1}\cap\mathbb{T}_{\beta'}$ is the top generator,
and the  class $\square$ has a unique holomorphic representative. 
The top generator $\Theta$ may be used to define 
$$\Ff_0^\spinc:\ov\CFKT(K_1,\spinc)\lra \ov\CFKT(K,\spinc).$$
The contribution 
of the degenerations of type $5$ thus corresponds to the coefficient of $\y$ in $\Ff_0^\spinc(\x)$. 
The map on homology induced by  $\Ff_0^\spinc$ coincides with the map used 
in the splicing formula of \cite{Ef-splicing}.
\\

Define the maps $F_0^\spinc:M(f_1^\spinc)\ra \ov\CFKT(K,\spinc)$ and 
$F_\infty^\spinc:M(f_0^\spinc)\ra M(f_1^\spinc)$
by 
\begin{displaymath}
\begin{split}
&F_0^\spinc(\x_1,\x_2):=I^\spinc(\x_1),\ \ \ \ \ \ \ \ \ \ \ \ \ \ 
\text{where }\begin{cases}
\x_1\in \ov\CFKT(K_m,\spinc)\oplus\ov\CFKT(K_m,\spinc+m-1)\\
\x_2\in\ov\CFT(Y).\end{cases}\\
&F_\infty^\spinc(\x_1,\x_2):=(G_\infty^\spinc(\x_1),-\x_2),\ \ \text{where }\begin{cases}
\x_1\in \ov\CFKT(K_m,\spinc)\oplus\ov\CFKT(K_m,\spinc+m)\\
\x_2\in\ov\CFT(Y)\end{cases}
\end{split}
\end{displaymath}
With this notation fixed, the outcome of the above observations, together with 
Lemma~\ref{lem:pentagon} and Lemma~\ref{lem:square} is the following theorem.

\begin{thm}\label{thm:commutative-diagram}
With the above notation fixed, the following diagram is commutative, upto chain homotopy 
\begin{equation}\label{eq:commutative-diagram}
\begin{diagram}
\ov\CFKT(K_0,\spinc)&\rTo^{\Ff_\infty^\spinc}&\ov\CFKT(K_1,\spinc)&\rTo^{\Ff_0^\spinc}&
\ov\CFKT(K,\spinc)\\
\dTo{\imath_0^\spinc}&&\dTo{\imath_1^\spinc}&&\dTo{Id}\\
M(f_0^\spinc)&\rTo{F_\infty^\spinc}&M(f_1^\spinc)&\rTo{F_0^\spinc}&\ov\CFKT(K,\spinc)
\end{diagram}.
\end{equation}
\end{thm}
\begin{proof}
By the discussion preceding the theorem, we have
$$\Ff_0^\spinc-F_0^\spinc\circ \imath_1^\spinc=d\circ Q^\spinc+Q^\spinc\circ d.$$
This proves the commutativity of the right-hand-side square upto chain homotopy. 
To prove the commutativity of the left-hand-side square, define 
\begin{displaymath}
\begin{split}
&R^\spinc:\ov\CFKT(K_0,\spinc)\lra M(f_1^\spinc),\ \ \ 
R^\spinc(\x):=(-H^\spinc(\x),P^\spinc(\x)).
\end{split}
\end{displaymath} 
We thus find
\begin{displaymath}
\begin{split}
(d\circ R^\spinc+ &R^\spinc\circ d)(\x)
=d(-H^\spinc(\x),P^\spinc(\x))+(R^\spinc\circ d)(\x)\\
&=\big((G_\infty^\spinc\circ g_0^\spinc-g_1^\spinc\circ \Ff_\infty^\spinc)(\x),
(d\circ P^\spinc+P^\spinc\circ d-f_1^\spinc\circ H^\spinc)(\x)\big)\\
&=-\big((g_1^\spinc\circ \Ff_\infty^\spinc-G_\infty^\spinc\circ g_0^\spinc)(\x),
(H_{h_1}^\spinc\circ \Ff_\infty^\spinc+H_{h_0}^\spinc)(\x)\big)\\
&=(F_\infty^\spinc\circ \imath_0^\spinc-\imath_1^\spinc\circ \Ff_\infty^\spinc)(\x).
\end{split}
\end{displaymath}
The second equality follows from Lemma~\ref{lem:square}, while 
the third equality follows from Lemma~\ref{lem:pentagon}. This observation 
completes the proof of Theorem~\ref{thm:commutative-diagram}.
\end{proof}

%\newpage
\section{Surgery and splicing formulas for knots}\label{sec:splicing}
\subsection{Surgery formulas}
Theorem~\ref{thm:mapping-cone} implies that 
$\ov{\CFKT}(K_n,\spinc)$ is quasi-isomorphic, for $m$ sufficiently large, 
to the mapping cone of the chain map 
$$f_n^\spinc:\ov\CFKT(K_{n+m},\spinc)\oplus\ov\CFKT(K_{n+m},\spinc+m)\lra \ov\CFT(Y).$$
When the curve $\lambda_{n+m}$ is very close to the juxtaposition of $\lambda$ and 
$(n+m)\beta_g$, and it cuts $\beta_g$ almost in the middle of the winding region,
this mapping cone has a particularly easy description, which is described below.\\

With the above choice we may assume 
that associated with every generator $\x$ for the Heegaard diagram 
$(\Sig,\alphas,\betas;u)$, which in turn is a generator of $\ov\CFT(Y)$, we obtain 
$n+m$ generators for $(\Sig,\alphas,\betas_{n+m};u,v)$. These $n+m$ 
generators will be denoted by
$\x_{1-l},\x_{2-l},...,\x_{m+n-l}$,
where $l=\lfloor m/2\rfloor$ and 
$\x_i$ is on the left of $\beta_g$ if $i<0$ and is on the right of $\beta_g$ otherwise.
The rest of generators for the Heegaard diagram 
$(\Sig,\alphas,\betas_{n+m};u,v)$ are in correspondence with the generators $\y$ of 
$(\Sig,\alphas,\betas_0;u,v)$. Every such generator will be denote by $\ov\y$. With this
notation fixed we have
\begin{displaymath}
\spinc(\x_i)=\begin{cases}\spinc(\x)+i\ \ &\text{if}\ i\geq 0\\
\spinc(\x)+n+m+i\ \ &\text{if}\ i<0\end{cases}\ \ \ \text{and} \ \ \ 
\spinc(\ov\y)=\spinc(\y)+n+\left\lceil\frac{m}{2}\right\rceil. 
\end{displaymath}   
Restricting our attention to the relative $\SpinC$ classes $\spinc$ and 
$\spinc+m$ we find
\begin{displaymath}
\begin{split}
&\ov\CFKT(K_{n+m},\spinc)=\left\langle \x_{\spinc-\spinc(\x)}\ \big|\ 
\x\in\Ta\cap\Tb\ \ \text{and}\ \ \spinc(\x)\leq \spinc\right\rangle,\\
&\ov\CFKT(K_{n+m},\spinc+m)=\left\langle \x_{\spinc-\spinc(\x)-n}\ \big|\ 
\x\in\Ta\cap\Tb\ \ \text{and}\ \ \spinc(\x)> \spinc-n\right\rangle,
\end{split}
\end{displaymath}
If the curve $\lambda_{n+m}$ is sufficiently close to the juxtaposition 
$\lambda\star (m+n)\beta_g$ the first complex is identified with the sub-complex
$$\left\langle \x\ \big|\ \x\in\Ta\cap\Tb\ \ \text{and}\ \ \spinc(\x)\leq 
\spinc\right\rangle$$
of $\ov\CFT(\Sig,\alphas,\betas;u)$, while the restriction of the map $f^\spinc$ 
to $\ov\CFKT(K_{n+m},\spinc)$ is identified with the inclusion of the aforementioned 
sub-complex in $\ov\CFT(Y)$. Similarly, the second complex is identified with the 
sub-complex $$\left\langle \x\ \big|\ \x\in\Ta\cap\Tb\ \ \text{and}\ \ 
\spinc(\x)> \spinc-n\right\rangle$$
of $\ov\CFT(\Sig,\alphas,\betas;s)$ 
while the restriction of the map $\ovl{f}^\spinc$ 
to  $\ov\CFKT(K_{n+m},\spinc+m)$ is identified with the inclusion of the
 aforementioned sub-complex in $\ov\CFT(\Sig,\alphas,\betas;s)$.\\

Let $C=C_K$ denote the $\Z\oplus\Z$-filtered chain complex generated by triples
$[\x,i,j]$ with $\x\in\Ta\cap\Tb,\ i,j\in\Z$ and $\spinc(\x)-i+j=0$. The differential 
of $C$ is defined by 
\begin{displaymath}
\begin{split}
d[\x,i,j]&=\sum_{\y\in\Ta\cap\Tb}\sum_{\phi\in\pi_2^1(\x,\y)}\#\left(
\ov\Mod(\phi)\right)[\y,i-n_u(\phi),j-n_s(\phi)]\\
&=\sum_{a,b=0}^\infty[d^{a,b}(\x),i-a,j-b].
\end{split}
\end{displaymath} 
Since $d\circ d=0$ we conclude that $d^{0,0}\circ d^{0,0}=0$, while 
\begin{equation}\label{eq:dod=0}
\begin{split}
&d^{0,1}\circ d^{0,0}+d^{0,0}\circ d^{0,1}=0,\\ 
&d^{1,0}\circ d^{0,0}+d^{0,0}\circ d^{1,0}=0\ \ \ \text{and}\\
&d^{1,1}\circ d^{0,0}+d^{0,0}\circ d^{1,1}
+d^{0,1}\circ d^{1,0}+d^{1,0}\circ d^{0,1}=0.
\end{split}
\end{equation}
Following \cite{OS-surgery} 
(or the notation set in the introduction) $\ov\CFT(Y)$ 
is identified as $C\{j=0\}$, while $\ov\CFKT(K_{n+m},\spinc)$ and 
$\ov\CFT(K_{n+m},\spinc+m)$ are identified with 
$$C\{i\leq \spinc,j=0\}\ \ \text{and}\ \ C\{i=0,j\leq n-\spinc-1\},$$
respectively. 
There is a  chain homotopy equivalence $\taub$  
from $C\{i=0\}$ to $C\{j=0\}$.
The following is thus just a re-statement of Theorem~\ref{thm:mapping-cone}.

\begin{thm}\label{thm:surgery-formula}
Fix the above notation and a class $\spinc\in\Z=\RelSpinC(Y,K)$. 
Then $\ov\CFKT(K_n,\spinc)$ is quasi-isomorphic to the mapping cone 
$M(i^\spinc_n)$ of 
\begin{displaymath}
\begin{split}
&i^\spinc_n:C\{i\leq \spinc,j=0\}\oplus C\{i=0,j\leq n-\spinc-1\}\lra C\{j=0\},\\
&i^\spinc_n([\x,i,0],[\y,0,j]):=[\x,i,0]+\taub[\y,0,j].
\end{split}
\end{displaymath}
\end{thm}

\subsection{The bypass homomorphisms}
We now turn to understanding the maps $F_0^\spinc$ and $F_\infty^\spinc$
(which will be called the {\emph{bypass homomorphisms}}) under the 
above identifications. To understand $F_0^\spinc$, one should identify $I^\spinc$ on 
$$\ov\CFKT(K_m,\spinc)\oplus\ov\CFKT(K_m,\spinc+m-1)=
C\{i\leq \spinc,j=0\}\oplus C\{i=0,j\leq -\spinc\}.$$
Let $\x\in\Ta\cap\Tb$,  $\x_i$ be the corresponding generator in $\ov\CFKT(K_m)$ 
and  suppose that $\square \in\pi_2^{-1}(\x_i,\Theta_{f_1},\Theta_\infty,\y;u,v,z)$ 
contributes to $I^\spinc$.
Looking at local coefficients implies that $i=-1$. In particular, 
$\spinc(\x)=\spinc(\y)=\spinc$ and $\x_{-1}$ corresponds to the generator 
$[\x,0,-\spinc]\in C\{i=0,j\leq-\spinc\}$.
There is a particular square class with very small domain which connects 
$\x_{-1},\Theta_{f_1},\Theta_\infty$ and $\x$ and has non-trivial contribution to 
$I^\spinc$. Considering the energy filtration and modifying 
$\ov\CFKT(K,\spinc)=C\{i=0,j=-\spinc\}$ by the chain map 
$I^\spinc|_{C\{i=0,j=-\spinc\}}$
which is a change of basis , we may thus assume that  
$F_0^\spinc$ is induced by projecting
the factor $C\{i=0,j\leq -\spinc\}$ in the mapping cone of $\imath^\spinc_1$ over 
the quotient complex $C\{i=0,j=-\spinc\}=\ov\CFKT(K,\spinc)$.\\

In order to study the map $F^\spinc_\infty$ we need to understand the map 
$$G^\spinc:\ov\CFKT(K_0,\spinc)\lra \ov\CFKT(K_m,\spinc+m-1).$$
Local considerations imply that
for a triangle class $\Delta\in\pi_2^0(\x,\Theta_{g_0}',\y;u,v,w)$ which has non-trivial 
contribution to $G^\spinc$ we have $\y=\z_i$ with $\z\in\Ta\cap\Tb$ and $i\leq -2$.
Every such $\Delta$ corresponds to a triangle class 
$\Delta'\in\pi_2^0(\x,\Theta_{g_0},\z_{i+1};u,v,w)$, and if $\lambda_m$ is 
sufficiently close to the juxtaposition $\lambda\star m\beta_g$ and $m$ is 
sufficiently large the moduli spaces $\Mod(\Delta)$ and $\Mod(\Delta')$ may 
in fact be identified. Note that $\spinc(\z_{i+1})=\spinc(\z_i)+1=\spinc+m$ and that
these latter disk classes $\Delta'$ are the disk classes which contribute to 
the holomorphic triangle map $g_{0,2}^\spinc$.
The image of $G^\spinc$ is thus in %the sub-complex
$$C\{i=0,j\leq -\spinc-1\}\subset \ov\CFKT(K_m,\spinc+m-1)=C\{i=0,j\leq -\spinc\},$$
and if 
\begin{displaymath}
\begin{split}
J^\spinc:\ov\CFKT(K_m,\spinc +m) &=C\{i=0,j\leq -\spinc-1 \}\\ 
&\lra \ov\CFKT(K_m,\spinc+m-1)=C\{i=0,j\leq -\spinc\}
\end{split}
\end{displaymath}
denotes the inclusion, $G^\spinc(\x)=J^\spinc(g_{0,2}^\spinc(\x))$. 
This  implies the following theorem.

\begin{thm}\label{thm:maps}
Under the identification of $\ov\CFKT(K_\bullet,\spinc)$ with 
$M(i_\bullet^\spinc)$ for $\bullet=0,1$, $F_\infty^\spinc$ is given by the inclusion of
 $M(i_0^\spinc)$ in  $M(i_1^\spinc)$ as a sub-complex,
while $F_0^\spinc$ is given by the quotient map. In particular, we have a short 
exact sequence
\begin{displaymath}
\begin{diagram}
0&\rTo & M(i_0^\spinc) &\rTo{F^\spinc_\infty=\hookrightarrow}&
M(i_1^\spinc)&\rTo{F_0^\spinc\ }& \ov\CFKT(K,\spinc)=\frac{M(i_1^\spinc)}{M(i_0^\spinc)}
&\rTo&0.
\end{diagram}
\end{displaymath}
\end{thm}
 
 Theorem~\ref{thm:maps} implies that the second row in %the diagram  of
  (\ref{eq:commutative-diagram}) is part of a short exact sequence.
 The discussion preceding Theorem 4.6 in \cite{Ef-splicing} implies that 
 the initial Heegaard diagram may be chosen 
 so that the first row is also completed to a short exact sequence.
 We thus have the following commutative diagram (upto chain homotopy):
\begin{equation}\label{eq:commutative-diagram-2}
\begin{diagram}
0&\rTo&\ov\CFKT(K_0,\spinc)&\rTo^{\ \Ff_\infty^\spinc\ }&
\ov\CFKT(K_1,\spinc)&\rTo^{\ \Ff_0^\spinc\ }&\ov\CFKT(K,\spinc)&\rTo &0\\
&&\dTo{\imath_0^\spinc}&&\dTo{\imath_1^\spinc}&&\dTo{Id}&&\\
0&\rTo& M(i_0^\spinc)&\rTo{F_\infty^\spinc}&M(i_1^\spinc)&\rTo{F_0^\spinc\ }&
\ov\CFKT(K,\spinc)&\rTo &0
\end{diagram}.
\end{equation}
 In particular, in the level of homology,  the connecting 
homomorphism of the short exact sequence  in the 
second row of (\ref{eq:commutative-diagram-2}) is identified with the connecting 
homomorphism $\Ff_1^\spinc$ of the first row, which is used in the splicing 
formula of \cite{Ef-splicing}.
A completely similar argument identifies $\ovl\Ff_\infty^\spinc$ with the inclusion map
$\ovl{F}_\infty^\spinc$ from $M(i_0^{\spinc-1})$ to $M(i_1^\spinc)$ and 
$\ovl\Ff_0^\spinc$
with the quotient map $\ovl{F}_0^\spinc$ to $\ov\CFKT(K,\spinc)$, while 
$\ovl\Ff_1^\spinc$ is identified with the connecting homomorphism of the short 
exact sequence
\begin{equation}\label{eq:SES-2}
\begin{diagram}
0&\rTo& M(i_0^{\spinc-1})&\rTo{\ovl{F}_\infty^\spinc}&M(i_1^\spinc)&
\rTo{\ovl{F}_0^\spinc}&
\ov\CFKT(K,\spinc)&\rTo &0
\end{diagram}.
\end{equation}

\begin{proof}(of Theorem~\ref{thm:BFH}) 
For $\bullet=0,1$, let us define
\begin{displaymath}
C_\bullet(K)=\bigoplus_{\spinc\in\RelSpinC(Y,K)}C_\bullet(K,\spinc),\ \ \ \text{where }
C_\bullet(K,\spinc)=M(i_\bullet^\spinc).
\end{displaymath}
 Let $C_\infty(K)=\bigoplus_\spinc C_\infty(K,\spinc)$, where 
$C_\infty(K,\spinc)=C\{i=\spinc,j=0\}$. The chain maps
$F_\infty, \ovl{F}_\infty \co C_0(K)\ra C_1(K)$ and 
$F_0, \ovl{F}_0 \co C_1(K)\ra C_\infty(K)$ sit in the  short exact sequences 
\begin{displaymath}
\begin{diagram}
0&\rTo{}&C_0(K)&\rTo{F_\infty}&C_1(K)&\rTo{F_0}&C_\infty(K)&\rTo &0\ \ 
&\text{and}\\
0&\rTo{}&C_0(K)&\rTo{\ovl{F}_\infty}&C_1(K)&\rTo{\ovl{F}_0}&C_\infty(K)&\rTo 
&0\ \ &\\
\end{diagram}
\end{displaymath}
The maps induced by $F_\bullet$ and $\ovl{F}_\bullet$ are 
$\pphi_\bullet$ and $\pphibar_\bullet$, respectively. Thus, Proposition 7.2 from 
\cite{Ef-splicing} may be applied here to complete the proof of Theorem~\ref{thm:BFH}.
\end{proof}

\section{The linear algebra of bypass homomorphisms}
\subsection{Nilpotent compositions}
Let $K$ be a knot inside the homology sphere $Y$ and let us continue to use the notation
of the previous sections.
\begin{lem}\label{lem:map-Y}
Let $\x\in\ov\CFKT(K,\spinc)$ be a closed element and $[\x]$ denote the class represented 
by $\x$ in $\ov\HFKT(K,\spinc)$. Then
\begin{displaymath}
(\ovl{F}^{\spinc-1}_0\circ {F}^{\spinc-1}_\infty\circ \ovl{F}^\spinc_1)[\x]=[d^{1,0}(\x)]
\ \ \ \text{and}\ \ \ 
({F}^{\spinc+1}_0\circ \ovl{F}^{\spinc+1}_\infty\circ {F}^\spinc_1)[\x]=[d^{0,1}(\x)].
\end{displaymath}
\end{lem}
\begin{proof}
Since $\ovl{F}_1^\spinc$ is the connecting homomorphism associated with the short 
exact sequence (\ref{eq:SES-2}), in order to compute $\ovl{F}^\spinc_1[\x]$
note that $\x$ is the image of $([\x,\spinc,0],0,0)\in M(i_1^\spinc)$ under 
the quotient map. 
The differential of   $M(i_1^\spinc)$ takes this  element to 
$$\left(\sum_{i=0}^\infty[d^{i,0}(\x),\spinc-i,0],0,[\x,\spinc,0]\right)
\in M(i_1^\spinc).$$ Since 
$d^{0,0}(\x)=0$ this latter element is in $M(i_0^{\spinc-1})$. 
$F^{\spinc-1}_\infty$ is 
the inclusion, thus
$$\left({F}^{\spinc-1}_\infty\circ \ovl{F}^\spinc_1\right)[\x]=
\left(\sum_{i=1}^\infty[d^{i,0}(\x),\spinc-i,0],0,[\x,\spinc,0]\right)\in M(i_1^{\spinc-1})
$$
The projection map $\ovl{F}^{\spinc-1}_0$  takes this latter element to the closed element
$d^{1,0}(\x)$ in $\ov\CFKT(K,\spinc-1)$.
%completing the proof of the first claim. 
The second claim is proved  similarly.
\end{proof}
\begin{cor}\label{cor:nilpotent}
For every relative $\SpinC$ class $\spinc$ the map
$$\ovl{\Ff}_0\circ {\Ff}_\infty\circ \ovl{\Ff}_1\circ {\Ff}_0
\circ \ovl{\Ff}_\infty\circ {\Ff}_1\big|_{\ov\HFKT(K,\spinc)}
:\ov\HFKT(K,\spinc)\lra \ov\HFKT(K,\spinc)$$
is nilpotent.
\end{cor}
\begin{proof}
It suffices to show that 
$F=\ovl{F}_0\circ {F}_\infty\circ \ovl{F}_1\circ {F}_0\circ \ovl{F}_\infty\circ {F}_1$
is nilpotent. However, by Lemma~\ref{lem:map-Y}, for $\x\in\ov\HFKT(K,\spinc)$
we have 
\begin{displaymath}
\begin{split}
&F[\x]=[d^{1,0}(d^{0,1}(\x))]\\
\Rightarrow\ \ \ &F^n[\x]=\left[\left(d^{1,0}\circ d^{0,1}\right)^n(\x)\right]
=\left[\left(\left(d^{1,0}\right)^n\circ  \left(d^{0,1}\right)^n\right)(\x)\right],
\end{split}
\end{displaymath}
where the last equality follows by an inductive use of  (\ref{eq:dod=0}).
Since $[(d^{0,1})^n(\x)]$ is in $\ov\HFKT(K,\spinc+n)$ and  for large values of $n$
$\ov\HFKT(K,\spinc+n)$ is trivial, it follows that $F^n$ is trivial  if $n$ is sufficiently  
large  (e.g. if $n>2g$ where $g$ is the genus of $K$).
\end{proof}

\subsection{Block decomposition for bypass homomorphisms}
Let us assume that the chain complex $C$ is defined from the Heegaard diagram 
$(\Sig,\alphas,\betas;u,v)$. Changing the role of the two punctures gives the duality maps
\begin{displaymath}
\begin{split}
\tau_\bullet=\tau_\bullet(K):\Hbb_\bullet(K)\lra \Hbb_\bullet(K),\ \ \  \ \
\bullet\in\{0,1,\infty\}.
\end{split}
\end{displaymath}
These duality maps take $\Hbb_\bullet(K,\spinc)$ to $\Hbb_\bullet(K,-\spinc)$
if $\bullet=1,\infty$, and to $\Hbb_0(K,1-\spinc)$ when $\bullet=0$.
Furthermore, we have $\tau_\bullet\circ \tau_\bullet=Id$. 
Following the notation of \cite{Ef-splicing}, in a basis for 
$\Hbb_\bullet(K)$ where $\pphi_\bullet$ takes the block form 
$\colvec{0&0\\ I&0}$, we assume that 
\begin{equation}\label{eq:tau}
\tau_\bullet
=\colvec[1]{A_\bullet & B_\bullet\\ C_\bullet & D_\bullet}
\ \ \ \bullet\in\{0,1,\infty\}.
\end{equation}
As observed in \cite{Ef-splicing}, one can then compute
$$\pphibar_0=\tau_\infty\circ \pphi_0 \circ \tau_1,\ \ \pphibar_1=\tau_0\circ \pphi_1\circ
\tau_\infty\ \ \text{and}\ \ \pphibar_\infty=\tau_1\circ \pphi_\infty\circ \tau_0.$$
Define $X_\bullet=X_\bullet(K)$ by $X_0=B_1 B_0B_\infty,
X_1=B_\infty B_1 B_0$ and 
$X_\infty= B_0 B_\infty B_1$. 
We denote the rank of $F_\bullet$ by 
$a_\bullet=a_\bullet(K)$, for $\bullet=0,1,\infty$. 
Thus $a_1,a_\infty$ and $a_0+1$ have 
the same parity. Note that $B_0, B_1$ and 
$B_\infty$ are (respectively) matrices of size 
\begin{displaymath}
a_\infty\times a_1,\ a_0\times a_\infty\ \ \text{and}\ \ a_1\times a_0.
\end{displaymath}
\begin{lem}\label{lem:a-is-positive}
If $K$ is a knot of genus $g>0$ then 
$B_\bullet\neq 0$ for $\bullet\in\{0,1,\infty\}$.
In particular,
$a_\bullet>0$.
\end{lem}
\begin{proof}
Since $H_*(M(i_0^{g}))=0$ by Theorem~\ref{thm:surgery-formula},
the map 
$$F_0^{g}:\Hbb_1(K,g)\lra \Hbb_\infty(K,g)$$
is an isomorphism. From here and by duality
$\ovl{F}_0^{-g}$ is also an isomorphism.
Similarly, 
\begin{displaymath}
\begin{split}
&H_*\left(M\left(i_1^{-g}\right)\right)\simeq \ov\HFKT(K,-g)\ \  \text{and}\\ 
&H_*\left(M\left(i_0^{-g}\right)\right)\simeq 
\ov\HFKT(K,-g)\oplus \ov\HFKT(K,-g).
\end{split}
\end{displaymath}
Thus $F_\infty^{-g}$ is surjective, i.e. 
$F_0^{-g}$ is trivial, implying that:
\begin{itemize}
\item $\Ker(F_0)\setminus \Ker(\ovl{F}_0)$ is non-empty.
\item $\Image(\ovl{F}_0)\setminus \Image (F_0)$ is non-empty.
\end{itemize} 
Returning to the matrix presentations, the first claim above implies that
\begin{displaymath}
\begin{split}
\exists\ \colvec{a\\ b}\in \Hbb_1(K)& \ \ \text{s.t. }
\begin{cases}
0=\colvec{0&0\\ I&0}\colvec{a\\ b}\\
\\
0\neq \colvec{A_\infty& B_\infty\\ C_\infty& D_\infty}
\colvec{0&0\\ I&0}
\colvec{A_1 & B_1\\ C_1 & D_1}
\colvec{a\\ b}
\end{cases}
\end{split}
\end{displaymath}
Thus $a=0$ and $\colvec{B_\infty B_1 b\\
D_\infty B_1 b}\neq 0$. In particular, $B_1\neq 0$.
Similarly, the second claim above implies that
$\Ker(\ovl{F}_1)\setminus \Ker({F}_1)$ is non-empty and thus
 $B_\infty\neq 0$.\\
For non-triviality of $B_0$, choose $x\in \Hbb_\infty(K, g)$.
This element may be represented by $y=[x,g,0]\in C\{j=0\}$.
Thus, $d^{*,0}y\in C\{i<g, j=0\}$ and 
$\ovl{F}_1^{g}(x)=(d^{*,0}y,y,0)\in M(i_0^{g-1})$
is thus in the kernel of $\ovl{F}_\infty$. If 
${F}_\infty^{g-1}(\ovl{F}_1^{g}(x))=0$ then 
$\ovl{F}_1^{g}(x)={F}_1^{g-1}(x')$ for some 
$x'\in \Hbb_\infty(K,g-1)$. In other words, if we denote 
the dual of $[x',0,1-g]$ by $z\in C\{i\leq 1-g,j=0\}$, the above equality 
implies
\begin{displaymath}
\begin{split}
\exists\ 
\begin{cases}
\ovl{y} \in C\{i<g,j=0\},\\
 \ovl{z} \in C\{i<1-g,j=0\}
 \end{cases}
\ \ \text{s.t }\ \begin{cases}
d^{*,0}(y-\ovl{y})=0\\
d^{*,0}(z-\ovl{z})=0\\
(y-\ovl{y})+(z-\ovl{z})\ \text{is exact.}
\end{cases}
\end{split}
\end{displaymath} 
Note that $-\ovl{y}+z-\ovl{z}\in C\{i<g,j=0\}$ while $y$
represents a non-trivial element in the homology of the quotient 
$$C\{i=g,j=0\}
=\frac{C\{i\leq g,j=0\}}{C\{i<g,j=0\}}.$$
 Thus $(y-\ovl{y})+(z-\ovl{z})$ can not be exact,  and 
 $\Ker({F}_\infty)\setminus \Ker(\ovl{F}_\infty)$
 can not be trivial. From here, an argument similar to the preceding two cases 
  implies $B_0\neq 0$.
\end{proof}

\begin{lem}\label{lem:X-nilpotent}
With the above notation fixed, the three matrices $X_\bullet$ are all nilpotent for 
$\bullet\in\{0,1,\infty\}$. In particular, if the knot $K$ is non-trivial
both the kernel and the cokernel of $X_\bullet$ are non-trivial.
\end{lem}
\begin{proof}
%Let $\tau_\bullet=\colvec{A_\bullet&B_\bullet\\
%C_\bullet& D_\bullet}$ denote the duality map defined 
%in Subsection~\ref{subsec:re-statement}.
The first claim is a direct consequence of  Corollary~\ref{cor:nilpotent} once we represent 
$F_\bullet=F_\bullet(K)$ as $\colvec{0&0\\ I&0}$ and note 
that $\ovl{F}_\bullet=F_\bullet(K)$ are given by   
\begin{displaymath}
\begin{split}
&\ovl{F}_0=\tau_\infty {F}_0 \tau_1,\ \ \ 
\ovl{F}_1=\tau_0 {F}_1 \tau_\infty\ \ \ \text{and}\ \ \ 
\ovl{F}_\infty=\tau_1 {F}_\infty \tau_0,
\end{split}
\end{displaymath}
which implies 
$\ovl{F}_0 F_\infty \ovl{F}_1 F_0 
\ovl{F}_\infty F_1 
=\colvec{(X_1)^2 &0\\ \star &0}$.
Thus there is a positive integer $N$ so that $X_1^N=0$. As a 
consequence $X_\bullet^{N+1}=0$ for $\bullet=0,1,\infty$.
The second claim is a consequence of the first claim unless $a_\bullet=0$.
However, by Lemma~\ref{lem:a-is-positive},  $a_\bullet>0$.
\end{proof}

\begin{defn}\label{def:full-rank-knot}
The knot $K$ inside the homology sphere $Y$ is called {\emph{full-rank}} if all three 
matrices $B_0(K),B_1(K)$ and $B_\infty(K)$ are full rank.
\end{defn}

If $P_\bullet$ is an invertible $a_\bullet\times a_\bullet$
matrix and the matrices $Y_\bullet$ are arbitrary matrices of correct size,
we may choose a change of basis for either of $\Hbb_0(K),\Hbb_1(K)$ and 
$\Hbb_\infty(K)$ which is given by 
 the invertible matrices
\begin{equation}\label{eq:change-of-basis}
\begin{split}
\Pbb_0=\left(\begin{array}{cc} P_\infty & 0\\ Y_0& P_1
\end{array}\right),\ \
\Pbb_1=\left(\begin{array}{cc} P_0& 0\\ Y_1& P_\infty
\end{array}\right)\ \text{and}\  \
\Pbb_\infty=\left(\begin{array}{cc} P_1 & 0\\ Y_\infty & P_0 
\end{array}\right),
\end{split}
\end{equation}
respectively. The block forms $F_\bullet=\colvec{0&0\\ I&0}$ remain unchanged 
under such a change of basis. A simultaneous 
change of basis of the form illustrated in (\ref{eq:change-of-basis}) is  called
an {\emph{admissible}} change of basis. 
The following lemma %which gives a standard presentation for the matrices 
 %$\tau_\bullet(K)$ for full-rank knots $K$ 
 will be useful through our forthcoming discussions.

\begin{lem}\label{lem:standarndard-form}
Suppose that $K$ is a knot in a homology sphere and 
for $\bullet\in\{0,1,\infty\}$ let $\tau_\bullet$ denote $\tau_\bullet(K)$ and 
$X_\bullet$ denote the matrix $X_\bullet(K)$. Choose
$$(\circ,\bullet,*)\in\{(0,1,\infty),(1,\infty,0),(\infty,0,1)\}.$$
\begin{itemize}
\item[(1)] If $B_\circ(K),B_\bullet(K)$ are injective and $B_*(K)$ is surjective, after 
an admissible change of basis we may assume that
\begin{equation}\label{eq:SF-1}
\begin{split}
&\tau_\circ=\colvec{0&0&\vline &I\\ 0&\star%J_*(K)
&\vline &0\\
\hline 
I&0&\vline &0}, \ 
\tau_\bullet=\colvec{0&0&0&\vline &I&0\\ 0&0&0&\vline &0&I\\
0&0&\star%J_\circ(K)
&\vline &0&0\\ \hline
I&0&0&\vline &0&0\\ 0&I&0&\vline &0&0}\ \
\text{and}\ \
\tau_*=\colvec{0&\vline &X_\bullet&  \star &  \star  \\ \hline
 \star  &\vline & \star & \star & \star \\ 
 \star  &\vline & \star & \star  & \star \\
 \star  &\vline & \star  & \star & \star  \\}
\end{split}
\end{equation}

\item[(2)] If $B_\circ(K),B_\bullet(K)$ are surjective and $B_*(K)$ is injective, after 
an admissible change of basis we may assume that
\begin{equation}\label{eq:SF-2}
\begin{split}
&\tau_\bullet=\colvec{0&\vline &0 &I\\ 
\hline 
0&\vline &\star%J_\circ(K) 
&0\\
I&\vline &0 &0}, \ \ 
\tau_\circ=
\colvec{0&0&\vline &0 &I&0\\ 0&0&\vline &0 &0&I\\ \hline
0&0&\vline &\star%J_*(K)
&0&0\\ 
I&0&\vline &0 &0&0\\ 0&I&\vline &0&0&0}\ \ \text{and}\ \ 
\tau_*=\colvec{
 \star & \star & \star  &\vline &  \star \\ 
 \star & \star & \star  &\vline & \star \\
 \star & \star & \star  &\vline & X_\circ\\  \hline
 \star & \star & \star  &\vline & 0 \\}
\end{split}
\end{equation}
\end{itemize}
\end{lem}
\begin{proof}
The proof consists of straight-forward linear algebra.
\end{proof}

\section{Splicing and homology sphere $L$-spaces}
\subsection{Special pairs}
 Given an arbitrary matrix $M$ denote the rank of $\Ker(M)$ by $k(M)$,
 denote the rank of $\Coker(M)$ by $c(M)$ and set $i(M)=k(M)+c(M)$. The
 matrices $M_1$ and $M_2$ are called {\emph{equivalent}} if $k(M_1)=k(M_2)$ and 
 $c(M_1)=c(M_2)$.  
 If  $M^\star\in M_{n_\star\times m_\star}(\Fbb)$  for $\star=1,2$ are a pair of matrices,
 $M^1\otimes M^2\in M_{n_1n_2\times m_1m_2}(\Fbb)$ is the associated  map from
 $\Fbb^{m_1m_2}=\Fbb^{m_1}\otimes \Fbb^{m_2}$ to $\Fbb^{n_1n_2}=\Fbb^{n_1}
 \otimes\Fbb^{n_2}$.\\ 

Let $Y=Y(K_1,K_2)$
denote the three-manifold obtained by splicing the complements of
$K_1\subset Y_1$ and $K_2\subset Y_2$, where $Y_1$ and $Y_2$ are homology spheres. 
For $\square\in\{A,B,C,D,X,\tau\}$, $\bullet\in\{0,1,\infty\}$ and $\star\in\{1,2\}$
 let $\square_\bullet^\star=\square_\bullet(K_\star)$. 
 Proposition 5.4 from 
 \cite{Ef-splicing} and the discussion following it give the following.
\begin{prop}\label{prop:splicing-formula}
If $K_i$ is a knot inside the homology sphere $Y_i$ for $i=1,2$, 
\begin{displaymath}
\rank\ \ov\HFT(Y(K_1,K_2);\Fbb)=i(\Dd(K_1,K_2)),
\end{displaymath}
where the matrix $\Dd(K_1,K_2)$ is given by
\begin{displaymath}
\colvec[.7]{
D_\infty^1B_1^1\otimes B_1^2A_0^2&B_1^1A_0^1\otimes I&
B_1^1B_0^1\otimes I&D_\infty^1 A_1^1\otimes B_1^2A_0^2&I\otimes B_1^2B_0^2&0\\
&&&\\
%%%%%%%%%%%%%%%%%%%%%%%%%%%%%%%%%%%%%%%%%%%%%%%%%%%
I\otimes B_\infty^2 B_1^2 &D_1^1A_0^1\otimes B_\infty^2 A_1^2 &
D_1^1B_0^1\otimes B_\infty^2 A_1^2& 0&B_0^1B_\infty^1\otimes I&
B_0^1 A_\infty^1\otimes I\\
&&&\\
%%%%%%%%%%%%%%%%%%%%%%%%%%%%%%%%%%%%%%%%%%%%%%%%%%%
I\otimes D_\infty^2 B_1^2& \begin{array}{c}I\otimes I+\\
D_1^1 A_0^1\otimes D_\infty^2 A_1^2\end{array}& D_1^1 B_0^1\otimes D_\infty^2 A_1^2 
&0&0&0\\
&&&\\
%%%%%%%%%%%%%%%%%%%%%%%%%%%%%%%%%%%%%%%%%%%%%%%%%%%
B_\infty^1 B_1^1\otimes I&0&I\otimes B_0^2B_\infty^2& B_\infty^1 A_1^1\otimes I
&\begin{array}{c} D_0^1 B_\infty^1\otimes B_0^2A_\infty^2\\ 
+X_1^1B_\infty^1\otimes B_0^2X_1^2\end{array}
&\begin{array}{c} D_0^1 A_\infty^1\otimes B_0^2A_\infty^2\\ 
+X_1^1A_\infty^1\otimes B_0^2X_1^2\end{array}\\
&&&\\
%%%%%%%%%%%%%%%%%%%%%%%%%%%%%%%%%%%%%%%%%%%%%%%%%%%
D_\infty^1 B_1^1\otimes D_1^2A_0^2&0&0&\begin{array}{c}I\otimes I+\\
D_\infty^1 A_1^1\otimes D_1^2A_0^2\end{array}&I\otimes D_1^2B_0^2&0\\
&&&\\
%%%%%%%%%%%%%%%%%%%%%%%%%%%%%%%%%%%%%%%%%%%%%%%%%%%
0&0&I\otimes D_0^2B_\infty^2 &0&\begin{array}{c} 
D_0^1 B_\infty^1\otimes D_0^2A_\infty^2\\
+X_1^1B_\infty^1\otimes D_0^2X_1^2\end{array}&\begin{array}{c}
I\otimes I+\\ D_0^1A_\infty^1\otimes D_0^2 A_\infty^2\\
+X_1^1 A_\infty^1\otimes D_0^2X_1^2\end{array}
},
\end{displaymath}
\end{prop}

\begin{defn}
The pair $(K_1,K_2)$ is called a {\emph{special pair}} if 
$$\ov\HFT(Y(K_1,K_2);\Fbb)=\Fbb.$$
\end{defn} 

Let us assume, throughout this section, that $(K_1,K_2)$ is a special pair, which is the 
case  if and only if $i(\Dd(K_1,K_2))=1$. 
Let  $k_\star^\bullet=k(B_\star^\bullet)$ and $c_\star^\bullet=c(B_\star^\bullet)$, 
for $\star\in\{0,1,\infty\}$ and $\bullet=1,2$.
Define $\imath:\{0,1,\infty\}\ra \{0,1,\infty\}$ by  
$\imath(0)=\infty,\imath(1)=1$ and $\imath(\infty)=0$.
 Let $\Dd=\Dd(K_1,K_2)$ and 
note that the cokernel of $\Dd$ includes a subspace $\CT(\Dd)$ and its kernel 
 includes a subspace $\KT(\Dd)$ which are isomorphic to 
\begin{displaymath}
\bigoplus_{\bullet\in\{0,1,\infty\}}\Coker(B_\bullet^1)\otimes 
\Coker(B_{\imath(\bullet)}^2)\ \ \text{and}\ \ 
\bigoplus_{\bullet\in\{0,1,\infty\}}\Ker(B_\bullet^1)\otimes 
\Ker(B_{\imath(\bullet)}^2)
\end{displaymath}
 respectively, and correspond  to the first, second and fourth  rows, 
 and to the first, third and fifth  columns, respectively.
Moreover, if $A_\infty^1\otimes D_0^2+D_0^1\otimes A_\infty^2=0$
(which may be assumed after an admissible change of basis 
if $c_\infty^1k_0^2=k_0^1c_\infty^2=0$)
the cokernel also includes a subspace  isomorphic to 
$\Coker(B_\infty^1)\otimes \Coker(B_\infty^2)$ 
and the kernel  includes a subspace isomorphic to 
$\Ker(B_0^1)\otimes \Ker(B_0^2)$.
Denote the ranks of $\KT(\Dd)$ and $\CT(\Dd)$ by $\ok(\Dd)$ and $\oc(\Dd)$,
respectively. 
Thus $ k(\Dd)+c(\Dd)\leq 1$ and 
\begin{displaymath}
\begin{split}
&\ok(\Dd)=\sum_{\bullet\in\{0,1,\infty\}}k_\bullet^1k_{\imath(\bullet)}^2\leq k(\Dd)\ \ \
\text{and}\ \ \
\oc(\Dd)=\sum_{\bullet\in\{0,1,\infty\}}c_\bullet^1c_{\imath(\bullet)}^2\leq c(\Dd).
\end{split}
\end{displaymath}

\begin{prop}\label{prop:all-special-pairs}
If $(K_1,K_2)$ is a special pair, then possibly after interchanging $K_1$ and $K_2$, 
one of the following is the case:
\begin{itemize}
\item[(G)] $K_1$ is full-rank.
\item[(S-1)] The matrix $B_0^2$ is invertible,
$B_0^1$ is surjective and $B_1^1$ and $B_\infty^2$ are injective.
\item[(S-2)] The matrix $B_0^2$ is invertible,
$B_0^1$ is injective and $B_1^1$ and $B_\infty^2$ are surjective.
\end{itemize}
\end{prop}  
\begin{proof}
We assume that $(K_1,K_2)$ is a special pair, while none of $K_1$ and $K_2$ is full-rank.
Let us first assume that both $\ok(\Dd)$ and $\oc(\Dd)$ are zero.
 From the above assumption we find
$k_\bullet^1k_{\imath(\bullet)}^2=c_\bullet^1c_{\imath(\bullet)}^2=0$ for 
$\bullet=0,1,\infty$. If $B_\bullet^1$ is not a full rank matrix then both $c_\bullet^1$ 
and $k_\bullet^1$ are non-zero. From here $k_{\imath(\bullet)}^2=c_{\imath(\bullet)}^2=0$,
i.e. $B_{\imath(\bullet)}^2$ is invertible. Since the parity of $a_0^2$ is different from the 
parity of $a_1^2$ and $a_\infty^2$, the matrices $B_1^2$ and $B_\infty^2$ can not be 
square matrices. Thus $\imath(\bullet)=0$ and $\bullet=\infty$. In other words, we 
conclude that $B_0^1$ and $B_1^1$ are full-rank and $B_0^2$ is invertible, 
while $B^2_\infty$ is not full-rank.
Similarly, we may conclude that $B_1^2$ is full-rank and $B_0^1$ is invertible, while 
$B_\infty^1$ is not full-rank.
Moreover, since $c_1^1c_1^2=k_1^1k_1^2=0$, precisely one of $B_1^1$ and $B_1^2$
is injective, and the other one is surjective. Without loosing on generality we may thus 
assume that:
\begin{itemize}
\item  $B_0^1$ and $B_0^2$ are invertible,
$B_1^1$ is injective and  $B_1^2$ is surjective.
\item None of $B_\infty^1$ and $B_\infty^2$ is full-rank.
\end{itemize} 
In particular, $k_\infty^1>c_\infty^1>0$ and $c_\infty^2>k_\infty^2>0$.
Since $B_0^1$ and $B_0^2$ are both invertible we may assume that $D_0^1=0$ and
$D_0^2=0$. From here the cokernel of $\Dd$ includes a subspace isomorphic 
to $\Coker(B_\infty^1)\otimes\Coker(B_\infty^2)$, which is of size 
$c_\infty^1c_\infty^2\geq 2$. This implies that $(K_1,K_2)$ is not special.\\
From this contradiction, 
we conclude that one of $\ok(\Dd)$ and $\oc(\Dd)$ is non-zero. Suppose 
that $\oc(\Dd)=1$ and $\ok(\Dd)=0$. For some $\bullet\in\{0,1,\infty\}$ we thus
have $c_\bullet^1=c_{\imath(\bullet)}^2=1$ while $k_\bullet^1k_{\imath(\bullet)}^2=0$ 
and for $\star\neq \bullet$ we have 
$c_\star^1 c_{\imath(\star)}^2=k_\star^1 k_{\imath(\star)}^2=0$.
Without loosing on generality we may assume that $k_\bullet^1=0$. Thus $B_\bullet^1$
is injective with a $1$-dimensional cokernel. In particular, the parity of the number of rows
and the number of columns for $B_\bullet^1$ are different, i.e. $\bullet\neq 0$.
Thus $c_0^1c_\infty^2=k_0^1k_\infty^2=0$. Since $B_\infty^2$ is not a square 
matrix, at least one of $c_\infty^2$ and $k_\infty^2$ is non-zero, implying that at least
one of $c_0^1$ and $k_0^1$ is zero, i.e. $B_0^1$ is full-rank. The assumption that
$K_1$ is not full-rank implies that $B_\star^1$ is not full-rank, where 
$\{\star\}=\{1,\infty\}\setminus\{\bullet\}$. From here $c_\star^1,k_\star^1>0$.
Together with $c_\star^1c_{\imath(\star)}^2=k_\star^1k_{\imath(\star)}^2=0$ this implies 
that $c_{\imath(\star)}^2=k_{\imath(\star)}^2=0$, i.e. $B_{\imath(\star)}^2$ is invertible.
Thus, $\imath(\star)=0, \star=\infty$ and $\bullet=1$. 
We thus conclude
\begin{itemize}
\item  $B_0^2$ is invertible,
$B_0^1$ is full-rank, $B_1^1$ is injective and $B_\infty^1$ is not full-rank.
\item $c_1^1=c_1^2=1$.
\end{itemize} 
Since $B_0^2$ is invertible, we may assume that $A_0^2=D_0^2=0$.
If $B_0^1$ is injective, we may also assume that $D_0^1=0$ and that 
$\Coker(\Dd)$  includes a subspace isomorphic to 
$\Coker(B_\infty^1)\otimes\Coker(B_\infty^2)$ and of size $c_\infty^1c_\infty^2$.
Since $c_\infty^1\neq 0$ we conclude that $B_\infty^2$ is surjective.
From here $a_\infty^2=a_1^2\leq a_0^2-1$ and 
$1-k_1^2=c_1^2-k_1^2=a_0^2-a_\infty^2\geq 1$.
We thus find $k_1^2=0$ and $K_2$ is full-rank, a contradiction.
Thus $k_0^1>0$ and $c_0^1=0$. From $k_0^1k_\infty^2=0$ we find
 $k_\infty^2=0$, i.e. $B_\infty^2$ is injective and the conditions of  (S-1) are 
 satisfied.
A similar argument reduces the case $\ok(\Dd)=1$ and $\oc(\Dd)=0$ to (S-2).
\end{proof}
\begin{prop}\label{prop:splicing-with-full-rank}
Given the pair of knots $(K_1,K_2)$ where $K_1$ is full-rank and 
$(\circ,\bullet,*)\in\{(0,1,\infty),(1,\infty,0),(\infty,0,1)\}$,
\begin{itemize}
\item[(K)] If $B_\circ^1,B_\bullet^1$ are injective and $B_*^1$ is surjective then 
\begin{displaymath}
\begin{split}
&c(\Dd)\geq c_\bullet^1c_{\imath(\bullet)}^2+c_\circ^1c_{\imath(\circ)}^2\ \ \text{and}\ \
k(\Dd)\geq k(X_\bullet^1)k(B_{\imath(*)}^2X_{\imath(\bullet)}^2).
\end{split}
\end{displaymath}
\item[(C)] If $B_\circ^1,B_\bullet^1$ are surjectiveand  $B_*^1$ is injective then 
\begin{displaymath}
\begin{split}
&k(\Dd)\geq k_\bullet^1k_{\imath(\bullet)}^2+k_\circ^1 k_{\imath(\circ)}^2\  \text{and}\ 
c(\Dd)\geq c(X_\bullet^1)c(X_{\imath(\bullet)}^2B_{\imath(*)}^2).
\end{split}
\end{displaymath}
\end{itemize} 
\end{prop}
\begin{proof} 
The first claim in either of cases (K) and  (C) is already observed in our earlier discussions. 
We  thus need to prove the second claim in each one of the above two cases.
The proofs  are very similar and we will only go through 
the proof for $(\circ,\bullet,*)=(0,1,\infty)$. In fact, the proof of claim (C) for 
$(\circ,\bullet,*)$ is almost identical to the proof of claim (K) for 
$(\imath(\bullet),\imath(\circ),\imath(*))$ because of the symmetry in the block 
presentation of $\Dd$.\\
We assume $(\circ,\bullet,*)=(0,1,\infty)$. In case (K),
after an admissible change of basis, we may assume that
$\tau_0(K_1), \tau_1(K_1)$ and $\tau_\infty(K_1)$ take the standard form 
of (\ref{eq:SF-1}). 
Since $D_0^1=D_1^1=A_\infty^1=0$, the $(3,2)$ entry  and the $(6,6)$ entry 
of the matrix $\Dd$ are both the identity matrix.
The matrix $\Dd$ is thus equivalent to the matrix
\begin{displaymath}
\colvec[.8]{\begin{array}{c}
D_\infty^1B_1^1\otimes B_1^2A_0^2+\\ B_1^1A_0^1\otimes D_\infty^2 B_1^2
\end{array} &
B_1^1B_0^1\otimes I&D_\infty^1 A_1^1\otimes B_1^2A_0^2&I\otimes B_1^2B_0^2\\
&&&\\
%%%%%%%%%%%%%%%%%%%%%%%%%%%%%%%%%%%%%%%%%%%%%%%%%%%
I\otimes B_\infty^2 B_1^2  &
0& 0&B_0^1B_\infty^1\otimes I\\
&&&\\
%%%%%%%%%%%%%%%%%%%%%%%%%%%%%%%%%%%%%%%%%%%%%%%%%%%
B_\infty^1 B_1^1\otimes I&I\otimes B_0^2B_\infty^2& B_\infty^1 A_1^1\otimes I
&X_1^1B_\infty^1\otimes B_0^2X_1^2\\
&&&\\
%%%%%%%%%%%%%%%%%%%%%%%%%%%%%%%%%%%%%%%%%%%%%%%%%%%
D_\infty^1 B_1^1\otimes D_1^2A_0^2&0&\begin{array}{c}I\otimes I+\\
D_\infty^1 A_1^1\otimes D_1^2A_0^2\end{array}&I\otimes D_1^2B_0^2\\
}.
\end{displaymath}
Replacing the block forms for $\tau_\star(K_1)$ gives the following presentation of the 
above matrix 
\begin{displaymath}
\colvec[.8]{
\star & \star & I\otimes I&0&0& \star & I\otimes B_1^2B_0^2&\star&\star\\
%%%%%%%%%%%%%%%%%%%%%%%%%%%%
\star& \star &0&0&0& \star &0&\star &\star\\
%%%%%%%%%%%%%%%%%%%%%%%%%%%%
\star& \star& 0&0&0& \star &0&\star&\star\\
%%%%%%%%%%%%%%%%%%%%%%%%%%%
\star &\star &0&0&0&\star& X_1^1\otimes I & \star & \star \\
%%%%%%%%%%%%%%%%%%%%%%%%%%%
\star &\star  &0&0&0&\star &0& \star&\star\\
%%%%%%%%%%%%%%%%%%%%%%%%%%%
\star &\star & I\otimes B_0^2B_\infty^2& 0&0&\star 
&X_1^1X_1^1\otimes B_0^2X_1^2&\star &\star\\ 
%%%%%%%%%%%%%%%%%%%%%%%%%%%
\star&\star&0& I\otimes I& 0&\star&\star&\star&\star\\
%%%%%%%%%%%%%%%%%%%%%%%%%%%
\star&\star&0& 0&I\otimes I&\star&0&\star&\star\\
%%%%%%%%%%%%%%%%%%%%%%%%%%%
\star& \star& 0&0&0&\star &0&\star& \star\\
}.
\end{displaymath}
After subtracting $I\otimes B_0^2B_\infty^2$ times the first row from  the fifth row, 
the identity matrices which appear in the entries $(1,3), (7,4)$ and $(8,5)$ of the above 
matrix become the only non-zero entries of their respective columns.
They  may thus be used for the cancellation of the third, the fourth and the fifth columns
against the first, the seventh and the eighth rows. We thus arrive at a $6\times 6$
matrix equivalent to $\Dd$, which is of the form   
\begin{displaymath}
\colvec[.8]{
\star& \star & \star &0&\star & \star \\
%%%%%%%%%%%%%%%%%%%%%%%
\star& \star& \star&0& \star & \star\\
%%%%%%%%%%%%%%%%%%%%%%%
\star & \star & \star & X_1^1\otimes I & \star & \star \\
%%%%%%%%%%%%%%%%%%%%%%%
 \star &\star & \star &0&  \star & \star \\
%%%%%%%%%%%%%%%%%%%%%%%
\star &\star & \star &(I+X_1^1X_1^1)\otimes B_0^2X_1^2&\star &\star\\ 
%&&\\
\star& \star&\star &0& \star & \star\\
}.
\end{displaymath}
Since the kernel of $\Dd$ includes a subspace 
which is isomorphic to  the kernel corresponding to the 
fourth column we find
$k(\Dd)\geq k(X^1_1)k(B_0^2X_1^2)$.\\
For case (C),
using Lemma~\ref{lem:standarndard-form} choose the standard block form of (\ref{eq:SF-2})
for $K_1$. In particular, $A_0^1,A_1^1$ and $D_\infty^1$ are all zero. The entries 
$(3,2)$ and $(5,4)$ of $\Dd$ are thus identity matrices which may be used
for cancellation. 
Add $B_\infty^1B_1^1\otimes B_0^2 X_1^2$ times the second row 
of the resulting matrix to its third row, add
$B_\infty^1B_1^1\otimes D_0^2 X_1^2$ times the second row 
 to the last row, and note that $B_1^1 D_1^1=0$
to  arrive at the following matrix, which is equivalent to $\Dd$: 
\begin{displaymath}
\colvec[.8]{
0&B_1^1B_0^1\otimes I&I\otimes B_1^2B_0^2&0\\
&&&\\
%%%%%%%%%%%%%%%%%%%%%%%%%%%%%%%%%%%%%%%%%%%%%%%%%%%
I\otimes B_\infty^2 B_1^2 &
D_1^1B_0^1\otimes B_\infty^2 A_1^2&B_0^1B_\infty^1\otimes I&
B_0^1 A_\infty^1\otimes I\\
&&&\\
%%%%%%%%%%%%%%%%%%%%%%%%%%%%%%%%%%%%%%%%%%%%%%%%%%%
B_\infty^1 B_1^1\otimes (I+X_0^2X_0^2)
&I\otimes B_0^2B_\infty^2
& D_0^1 B_\infty^1\otimes B_0^2A_\infty^2
&D_0^1 A_\infty^1\otimes B_0^2A_\infty^2\\
&&&\\
%%%%%%%%%%%%%%%%%%%%%%%%%%%%%%%%%%%%%%%%%%%%%%%%%%%
B_\infty^1B_1^1\otimes D_0^2 X_1^2B_\infty^2 B_1^2&I\otimes D_0^2B_\infty^2 & 
D_0^1 B_\infty^1\otimes D_0^2A_\infty^2
&\begin{array}{c}
I\otimes I+\\ D_0^1A_\infty^1\otimes D_0^2 A_\infty^2 \end{array}
}.
\end{displaymath}
Replacing the block forms of (\ref{eq:SF-2}) for $\tau_0(K_1),\tau_1(K_1)$ and 
$\tau_\infty(K_1)$ we arrive at a matrix of the form
\begin{displaymath}
\colvec[.8]{
0&0&0&0&I\otimes I& I\otimes B_1^2B_0^2 &0&0&0\\
%%%%%%%%%%%%%%%%%%%%%%%%%%%%%%%%%%%%%%%%%%%%%%%%%%%
 \star  & \star & \star & \star &0&\star &\star &\star &\star\\
%%%%%%%%%%%%%%%%%%%%%%%%%%%%%%%%%%%%%%%%%%%%%%%%%%%
 \star  & \star & \star & \star &0&\star &\star &\star &\star\\
%%%%%%%%%%%%%%%%%%%%%%%%%%%%%%%%%%%%%%%%%%%%%%%%%%%
 \star  & \star & \star & \star &0&\star &\star &\star &\star\\
%%%%%%%%%%%%%%%%%%%%%%%%%%%%%%%%%%%%%%%%%%%%%%%%%%%
 \star  & \star & \star & \star &0&\star &\star &\star &\star\\
%%%%%%%%%%%%%%%%%%%%%%%%%%%%%%%%%%%%%%%%%%%%%%%%%%%
0&X_0^1\otimes (I+X_0^2X_0^2)& 0&0& I\otimes B_0^2B_\infty^2 &0
& 0 &0 &0\\
%%%%%%%%%%%%%%%%%%%%%%%%%%%%%%%%%%%%%%%%%%%%%%%%%%%
 \star  & \star & \star & \star &0&\star &\star &\star &\star\\
%%%%%%%%%%%%%%%%%%%%%%%%%%%%%%%%%%%%%%%%%%%%%%%%%%%
 \star  & \star & \star & \star &0&\star &\star &\star &\star\\
%%%%%%%%%%%%%%%%%%%%%%%%%%%%%%%%%%%%%%%%%%%%%%%%%%%
 \star  & \star & \star & \star &0&\star &\star &\star &\star\\
},
\end{displaymath}
which is in turn equivalent to a matrix of the form
\begin{displaymath}
\colvec[.8]{
%%%%%%%%%%%%%%%%%%%%%%%%%%%%%%%%%%%%%%%%%%%%%%%%%%%
 \star  & \star & \star & \star &\star &\star &\star &\star\\
%%%%%%%%%%%%%%%%%%%%%%%%%%%%%%%%%%%%%%%%%%%%%%%%%%%
 \star  & \star & \star & \star &\star &\star &\star &\star\\
%%%%%%%%%%%%%%%%%%%%%%%%%%%%%%%%%%%%%%%%%%%%%%%%%%%
 \star  & \star & \star & \star &\star &\star &\star &\star\\
%%%%%%%%%%%%%%%%%%%%%%%%%%%%%%%%%%%%%%%%%%%%%%%%%%%
 \star  & \star & \star & \star &\star &\star &\star &\star\\
%%%%%%%%%%%%%%%%%%%%%%%%%%%%%%%%%%%%%%%%%%%%%%%%%%%
0&X_0^1\otimes (I+X_0^2X_0^2)& 0&0& I\otimes X_\infty^2 B_0^2&0&0&0\\
%%%%%%%%%%%%%%%%%%%%%%%%%%%%%%%%%%%%%%%%%%%%%%%%%%%
 \star  & \star & \star & \star &\star &\star &\star &\star\\
%%%%%%%%%%%%%%%%%%%%%%%%%%%%%%%%%%%%%%%%%%%%%%%%%%%
 \star  & \star & \star & \star &\star &\star &\star &\star\\
%%%%%%%%%%%%%%%%%%%%%%%%%%%%%%%%%%%%%%%%%%%%%%%%%%%
 \star  & \star & \star & \star &\star &\star &\star &\star\\
},
\end{displaymath}
In particular, we conclude $c(\Dd)\geq c(X_0^1)c(X_\infty^2B_0^2)$.
This completes the proof of case (C) when $(\circ,\bullet,*)=(0,1,\infty)$.
\end{proof}
\subsection{The special cases (S-1) and (S-2)}
\begin{lem}\label{lem:special-cases}
If $(K_1,K_2)$ is a special pair of type (S-1) or (S-2) then one of the knots $K_1$ or 
$K_2$ is trivial.
\end{lem}
\begin{proof}
Suppose otherwise that $(K_1,K_2)$ is a special pair of type (S-1)
and that both $K_1$ and $K_2$ are non-trivial. After an admissible 
change of basis,  assume that
\begin{equation}\label{eq:SF-3}
\begin{split}
&\tau_0^2=\colvec{0&0&\vline &I&0\\
0&0&\vline &0&I\\
\hline 
I&0&\vline &0&0\\
0&I&\vline &0&0}, \ \ 
\tau_\infty^2=\colvec{0&0&\vline &I\\ 
0&\star&\vline &0\\ \hline
I&0&\vline &0}\ \ 
\text{and}\ \ \ \ 
\tau_1^2=\colvec{\star &\vline &X_\infty^2&  \star   \\ \hline
 \star  &\vline & \star & \star  \\ 
 \star  &\vline & \star & \star   }
\end{split}
\end{equation}
In particular, $A_0^2,D_0^2$ and $D_\infty^2$ are zero. We may also assume 
that 
\begin{equation}\label{eq:SF-4}
\begin{split}
&\tau_0^1=\colvec{0&\vline &I&0\\
\hline 
I&\vline &0&0\\
0&\vline &0&\star}, \ \ 
\tau_1^1=\colvec{0&0&\vline &I\\ 
0&\star &\vline &0\\ \hline
I&0&\vline &0}\ \ 
\text{and}\ \ \ \ 
\tau_\infty^1=\colvec{\star& \star &\vline &X_\infty^1&  \star   \\ 
\star&\star&\vline&\star&\star\\
\hline
 \star&\star  &\vline & \star & \star  \\ 
 \star &\star &\vline & \star & \star   }
\end{split}
\end{equation}
In particular,
$A_0^1$ and $D_1^1$ are zero. The identity matrices which appear as 
entries $(3,2)$, $(5,4)$ and $(6,6)$ in $\Dd(K_1,K_2)$ 
 may be used for cancellation to obtain the equivalent matrix
\begin{displaymath}
\colvec[.8]{
0&B_1^1B_0^1\otimes I&I\otimes B_1^2B_0^2\\
&&&\\
%%%%%%%%%%%%%%%%%%%%%%%%%%%%%%%%%%%%%%%%%%%%%%%%%%%
I\otimes B_\infty^2 B_1^2  &
 0&B_0^1B_\infty^1\otimes I\\
&&&\\
%%%%%%%%%%%%%%%%%%%%%%%%%%%%%%%%%%%%%%%%%%%%%%%%%%%
B_\infty^1 B_1^1\otimes I&I\otimes B_0^2B_\infty^2
&\begin{array}{c} D_0^1 B_\infty^1\otimes B_0^2A_\infty^2\\ 
+X_1^1B_\infty^1\otimes B_0^2X_1^2\\+B_\infty^1A_1^1\otimes D_1^2B_0^2\end{array}
}.
\end{displaymath}
Subtracting $X_1^1B_\infty^1\otimes B_0^2B_\infty^2$ times the first row from the 
third row we arrive at the equivalent matrix
\begin{displaymath}
\colvec[.8]{
0&B_1^1B_0^1\otimes I&I\otimes B_1^2B_0^2\\
&&&\\
%%%%%%%%%%%%%%%%%%%%%%%%%%%%%%%%%%%%%%%%%%%%%%%%%%%
I\otimes B_\infty^2 B_1^2  &
 0&B_0^1B_\infty^1\otimes I\\
&&&\\
%%%%%%%%%%%%%%%%%%%%%%%%%%%%%%%%%%%%%%%%%%%%%%%%%%%
B_\infty^1 B_1^1\otimes I&(I+X_1^1X_1^1)\otimes B_0^2B_\infty^2
&\begin{array}{c} D_0^1 B_\infty^1\otimes B_0^2A_\infty^2\\ 
+B_\infty^1A_1^1\otimes D_1^2B_0^2\end{array}
}.
\end{displaymath}
Replacing the block forms of (\ref{eq:SF-3}) and (\ref{eq:SF-4}), 
the above matrix takes the form
\begin{displaymath}
\colvec[.8]{
0&\star&I\otimes I &0& I\otimes X_\infty^2 & \star&\star&\star\\
0&\star&0&0&0& \star&\star&\star\\
I\otimes X_\infty^2 &\star & 0&0& X_\infty^1\otimes I& \star&\star&\star\\
0&\star&0&0&0& \star&\star&\star\\
X_\infty^1\otimes I& \star& (I+X_\infty^1 X_\infty^1)\otimes I&0&0& \star&\star&\star\\
0&\star&0&0&0& \star&\star&\star\\
\star&\star&\star& I\otimes I&0& \star&\star&\star\\
0&\star &0&0&0& \star&\star&\star
}.
\end{displaymath}
Subtract $(I+X_\infty^1 X_\infty^1)\otimes I$ times the first row from the fifth row 
and use the identity matrices which appear as $(1,3)$ and $(7,4)$ entries of the above 
matrix for cancellation to arrive at the following equivalent matrix
\begin{displaymath}
\colvec[.8]{
0&\star&0& \star&\star&\star\\
I\otimes X_\infty^2 &\star & X_\infty^1\otimes I& \star&\star&\star\\
0&\star&0& \star&\star&\star\\
X_\infty^1\otimes I&\star&(I+X_\infty^1 X_\infty^1)\otimes X_\infty^2&\star&\star&\star\\
0&\star&0& \star&\star&\star\\
0&\star &0& \star&\star&\star
}.
\end{displaymath}
From the above presentation we conclude 
$$k(\Dd)\geq 2k(X_\infty^1)k(X_\infty^2)\geq 2.$$
This contradiction rules out the case (S-1). Ruling out the case (S-2) is similar.
\end{proof}

\section{Incompressible tori in homology spheres}\label{sec:incomp-torus}
\subsection{The main theorem}\label{subsec:main-thm}
\begin{thm}\label{thm:main}
Suppose that  $K_i$ is a non-trivial knot in the homology sphere $Y_i$ for $i=1,2$.  
Let $Y=Y(K_1,K_2)$ denote the three-manifold obtained by splicing the complements 
of $K_1$ and $K_2$. Then the rank of $\widehat{\text{\emph{HF}}}(Y)$ is bigger 
than one.
\end{thm}
\begin{proof}
Suppose otherwise that $Y$ is a $L$-space. Thus $(K_1,K_2)$ is a special pair.
By Proposition~\ref{prop:all-special-pairs} and Lemma~\ref{lem:special-cases}
we may assume that $K_1$ is full-rank. In particular, one of the cases (K) 
 or (C) from 
Proposition~\ref{prop:splicing-with-full-rank}  will happen. Note that in case 
(K) the kernel of $\Dd$ is necessarily non-trivial by Lemma~\ref{lem:X-nilpotent}, while 
in case (C) the cokernel of $\Dd$ is non-trivial.\\
Let us assume that (K) is the case. Thus $c(\Dd)=0$ and 
$k(X_\bullet^1)=k(B_{\imath(*)}^2X_{\imath(\bullet)}^2)=1$. 
Note that 
$\Ker(B_{\imath(*)}^2)\subset 
\Ker(B_{\imath(*)}^2X_{\imath(\bullet)}^2)$, which implies that 
either $B_{\imath(*)}^2$ is injective or $\Ker(B_{\imath(*)}^2)=
\Ker(B_{\imath(*)}^2X_{\imath(\bullet)}^2)$. If the latter happens, we find
\begin{displaymath}
\Ker(B_{\imath(*)}^2)=
\Ker(B_{\imath(*)}^2X_{\imath(\bullet)}^2)=
\Ker(B_{\imath(*)}^2X_{\imath(\bullet)}^2X_{\imath(\bullet)}^2)=\dots =\Ker(0),
\end{displaymath}
since $X_{\imath(\bullet)}^2$ is nilpotent by Lemma~\ref{lem:X-nilpotent}. 
Since $B_{\imath(*)}^2\neq 0$ this can not happen and we conclude that 
$B_{\imath(*)}^2$  is injective. \\
Let us first assume that $B_0^1$ is not invertible. Then $c_\bullet^1,c_\circ^1\neq 0$.
Since $c_\bullet^1c_{\imath(\bullet)}^2=c_\circ^1c_{\imath(\circ)}^2=0$ 
we conclude that $B_{\imath(\circ)}^2$ and $B_{\imath(\bullet)}^2$ are both surjective.
Thus $K_2$ is full-rank and by part (C) of Proposition~\ref{prop:splicing-with-full-rank}
$c(\Dd)>0$. This contradiction implies that $B_0^1$ is invertible. Moreover, 
the argument implies that $0\in\{\circ,\bullet\}$ and at least one of 
$c_{\imath(\circ)}^2$ and $c_{\imath(\bullet)}^2$ is trivial. It is easy to conclude 
from here that we are then either in case (S-1) or case (S-2) of 
Proposition~\ref{prop:all-special-pairs}, which are both excluded by 
 Lemma~\ref{lem:special-cases}. The contradiction rules out case (K) of
  Proposition~\ref{prop:splicing-with-full-rank}. Excluding the case (C) is completely similar.
\end{proof}

\begin{cor}
If the homology sphere $Y$ contains an incompressible torus then 
$\rank(\ov{\mathrm{HF}}(Y,\Fbb)>1$.
\end{cor}
\begin{proof}
If $Y$ contains an incompressible torus $T$, $T$ will be separating and there will be a pair of
curves $\lambda$ and $\mu$ on $T$ such that $\lambda$ is homologically 
trivial on one side of $T$ and $\mu$ is homologically trivial on the other 
side of $T$. Since $Y$ is a homology sphere, the intersection number of $\mu$ 
and $\lambda$ is one. Let $U_1$ and $U_2$ be the two components of
$Y-T$ and let $U_1$ be the component containing a surface which bounds 
$\lambda$. Capping off $\mu\subset T=\partial U_1$ by a disk and then 
gluing a three-ball gives a three-manifold $Y_1$.
The simple closed curve $\lambda$ represents a knot $K_1\subset Y_1$. 
Similarly capping off $\lambda\subset T=\partial U_2$ by a disk and then 
gluing a three-ball gives a three-manifold $Y_2$ and $\mu$ 
represents a knot $K_2\subset Y_2$. Both $Y_1$ and $Y_2$ are homology
spheres and $Y$ is obtained by splicing $K_1$ and $K_2$. Since $T$ is 
incompressible, both $K_1$ and $K_2$ are non-trivial and 
Theorem~\ref{thm:main} completes the proof of this corollary.
\end{proof}

%\newpage
\subsection{Applications}\label{subsec:applications}
We may use the relation between Khovanov
homology of a knot inside the standard sphere and the Heegaard
Floer homology of its branched double-cover, discovered by
Ozsv\'ath and Szab\'o \cite{OS-branched}, to show the non-triviality of
Khovanov homology for certain classes of knots. We emphasize again that 
the results presented here are all special cases of the the theorem of 
Kronheimer and Mrowka \cite{KM} that Khovanov homology is an unknot 
detector.

\begin{defn}
A prime knot $K\subset S^3$ is an $n${\emph{-string composite}} if there
is an embedded $2$-sphere intersecting the knot transversely which separates $(S^3,K)$
into prime $n$-string tangles. A $2$-string composite knot is called a 
\emph{doubly composite knot}.
\end{defn}
We refer the reader to \cite{Bleiler} for more on doubly composite and 
doubly prime knots, and only quote the following lemma from that paper:
\begin{lem}\label{lem:doubly-composite}
A prime knot $K\subset S^3$ is a doubly composite knot if and only if the double cover
$\Sig(K)$ of $S^3$ branched over the knot $K$ contains an incompressible torus $T$
which is invariant under the non-trivial covering translation and meets the fixed point set
of this map precisely in $4$ points, and separates $\Sig(K)$ into 
irreducible boundary irreducible pieces.
\end{lem}
\begin{cor}\label{cor:composite}
If the prime knot $K\subset S^3$ is doubly composite, 
the rank of its reduced Khovanov homology group
$\widetilde{\mathrm{Kh}}(K)$ is bigger than $1$.
\end{cor}
\begin{proof}
If $K$ is doubly composite, by Lemma~\ref{lem:doubly-composite} there exists 
an incompressible torus $T$ inside the three-manifold $\Sig(K)$. 
Thus the rank of $\ov{\mathrm{HF}}(\Sig(K),\Fbb)$ is bigger than $1$.
By the main theorem of \cite{OS-branched} there is a spectral sequence 
whose $E^2$-term consists of Khovanov's reduced homology $\widetilde{\mathrm{Kh}}(K)$
of the mirror of $K$ with coefficients in $\Fbb$ which converges to
 $\ov{\mathrm{HF}}(\Sig(K),\Fbb)$, and is of rank greater than $1$ by 
 Theorem~\ref{thm:main}. Thus the rank of 
 $\widetilde{\mathrm{Kh}}(K)$ is bigger than $1$  as well.
\end{proof}
Furthermore, if $K$ is a prime satellite knot, we will have an incompressible torus in the
complement of $K$. This torus gives an incompressible torus in the double cover $\Sig(K)$
of $S^3$ branched  over the knot $K$. Thus, Heegaard Floer homology of 
$\Sig(K)$ will be non-trivial. We thus have the following corollary:
\begin{cor}
If $K\subset S^3$ is a prime satellite knot
the rank of its reduced Khovanov homology group $\widetilde{\mathrm{Kh}}(K)$
 is greater than $1$.
\end{cor}
In fact, we may prove a slightly more general statement:
\begin{prop}\label{prop:pi-hyperbolic}
If the rank of the reduced Khovanov homology $\widetilde{\mathrm{Kh}}(K)$ 
of a non-trivial knot $K\subset S^3$ is one, the double  cover $\Sig(K)$ 
of $S^3$, branched over the knot $K$, is hyperbolic.
\end{prop}
\begin{proof}
Note that if a knot $K$ is doubly composite Corollary~\ref{cor:composite} implied 
that the rank of  $\widetilde{\mathrm{Kh}}(K)$ is bigger than $1$. 
Thus, $K$ has to be doubly prime.  By Thurston's orbifold geometrization 
theorem (see \cite{BoP} and \cite{CHK}) the branched double cover $\Sig(K)$ 
is a geometric manifold and there are three possible cases.\\
 1- $\Sig(K)$ is a Lens space and thus admits a spherical structure.
  If $\ov{\mathrm{HF}}(\Sig(K))$  is one dimensional, $\Sig(K)$ is forced to
   be the standard sphere and $K$ is trivial. Thus in this case,
 the rank of  $\widetilde{\mathrm{Kh}}(K)$ is bigger than $1$ only if $K$ is trivial. \\
 2- $\Sig(K)$ admits a Seifert fibration and $K$ is a Montesinos knot with at most 
 three rational tangles.
 If $\Sig(K)$ is not a homology sphere, $\widetilde{\mathrm{Kh}}(K)$
 is clearly different from $\Fbb$, and if it is a homology sphere which 
 admits a Seifert fibration and $\ov{\mathrm{HF}}(\Sig(K))=\Fbb$, we 
 know (see \cite{Raif} or \cite{Ef-Seifert}) that $\Sig(K)$ is either
 the standard sphere, or the Poincar\'e sphere.  Moreover, for $\Sig(K)$ to be the
 Poincar\'e sphere we should have $K=T(3,5)$, i.e. $K$ is the $(3,5)$-torus 
 knot, or equivalently
 $(-2,3,5)$-pretzel knot, which is $10_{124}$ in Rolfsen's 
 table (see \cite{Matt-Kh} and \cite{Rolf}).
  $\widetilde{\mathrm{Kh}}(T(3,5))$ has rank $7$ by direct computation \cite{Ko}.\\
 3- $\Sig(K)$ admits a hyperbolic structure which is invariant under the
 deck transformation.\\
 Having ruled out the first two possibilities, the proof  is complete.
\end{proof}
The knots $K$ with the property that $\Sig(K)$ admits a 
hyperbolic structure which is invariant under the involution of $\Sig(K)$ are 
called $\pi$-\emph{hyperbolic}. The hyperbolic structure comes
from a hyperbolic structure on $S^3-K$ which becomes a singular 
folding with angle $\pi$ around $K$.
Thus in particular, $\pi$-hyperbolic knots are hyperbolic.
\begin{cor}
Assuming Conjecture~\ref{conj},  
if the reduced Khovanov homology $\widetilde{\mathrm{Kh}}(K)$ for 
a knot $K\subset S^3$ is equal to $\Fbb$, $K$ is the unknot.
\end{cor}
\begin{proof}
Suppose $K$ is not the unknot.
By Proposition~\ref{prop:pi-hyperbolic}, if $\widetilde{\mathrm{Kh}}(K)=\Fbb$,  
the branched double
cover $\Sig(K)$ is  hyperbolic. Conjecture~\ref{conj} then implies that
 $\ov{\mathrm{HF}}(\Sig(K))$ is non-trivial, and by the correspondence of 
 \cite{OS-branched},
$$1=\rank\big(\widetilde{\mathrm{Kh}}(K)\big)\geq \rank(\ov{\mathrm{HF}}(\Sig(K)))>1.$$
This contradiction implies that  $\widetilde{\mathrm{Kh}}(K)\neq \Fbb$.
\end{proof}
% ----------------------------------------------------------------

\end{document}